\documentclass[a4paper,11pt,english]{smfart}

\usepackage{amssymb}
\usepackage{url}
\usepackage[T1]{fontenc}
\RequirePackage{calrsfs}
\DeclareSymbolFont{rsfscript}{OMS}{rsfs}{m}{b}
\DeclareSymbolFontAlphabet{\mathrsfs}{rsfscript}
\usepackage[utopia,expert]{mathdesign}
\usepackage{lscape}

\usepackage{enumitem}

\usepackage{todonotes}

\usepackage[vcentermath]{youngtab}

\usepackage[leftbars]{changebar}

\usepackage{palatino}
\usepackage{rotating}
\usepackage{graphicx}

\usepackage{tikz-cd}

\usepackage{enumerate}

\usepackage{amscd}

\usepackage{color}
\definecolor{shadecolor}{gray}{0.90}

\def\bfit{\bfseries\itshape}

\def\proofname{D\'emonstration}

\input xypic
\xyoption{all}
\xyoption{arc}

\makeindex

\def\indexname{Index des notations}

\newtheorem{theo}{Theorem}[section]
\def\thetheo{\thesection.\arabic{theo}}
\newtheorem{prop}[theo]{Proposition}

\newtheorem{lem}[theo]{Lemma}

\newtheorem{coro}[theo]{Corollary}

\newtheorem{defi}[theo]{Definition}

\renewcommand\theequation{\thetheo}
\def\equat{\refstepcounter{theo}\begin{equation}}
\def\endequat{\end{equation}}

\renewcommand\thetable{\thetheo}

\renewcommand\thesection{\arabic{section}}
\renewcommand\thesubsection{\thesection.\Alph{subsection}}

\def\contentsname{Table des mati\`eres}
\def\listfigurename{Liste des figures}
\def\listtablename{Liste des tables}
\def\appendixname{Appendice}

\def\refname{R\'ef\'erences}

    \def\CM{{\mathbb{C}}}

    \def\FM{{\mathbb{F}}}
  \def\gG{{\mathfrak g}}  
  \def\hG{{\mathfrak h}}

    \def\NM{{\mathbb{N}}}

\def\SG{{\mathfrak S}}

    \def\ZM{{\mathbb{Z}}}

    \def\BC{{\mathcal{B}}}
    \def\CC{{\mathcal{C}}}
    
    \def\EC{{\mathcal{E}}}
    \def\FC{{\mathcal{F}}}
\def\Gb{{\mathbf G}}    
\def\Hb{{\mathbf H}}    \def\HC{{\mathcal{H}}}

\def\Lb{{\mathbf L}}

    \def\OC{{\mathcal{O}}}
\def\Pb{{\mathbf P}}    \def\PC{{\mathcal{P}}}

\def\Xb{{\mathbf X}}    \def\XC{{\mathcal{X}}}
    \def\YC{{\mathcal{Y}}}
\def\Zb{{\mathbf Z}}    \def\ZC{{\mathcal{Z}}}

\def\Hrm{{\mathrm{H}}}    
\def\Irm{{\mathrm{I}}}

\def\Trm{{\mathrm{T}}}

    \def\XCB{{\boldsymbol{\mathcal{X}}}}
    \def\YCB{{\boldsymbol{\mathcal{Y}}}}
\def\Zrm{{\mathrm{Z}}}    \def\ZCB{{\boldsymbol{\mathcal{Z}}}}

\def\Ati{{\tilde{A}}}

\def\Qov{{\overline{Q}}}

\def\a{\alpha}
\def\b{\beta}
\def\g{\gamma}

\def\d{\delta}
\def\D{\Delta}
\def\e{\varepsilon}

\def\l{\lambda}
\def\L{\Lambda}

\def\o{\omega}
\def\O{\Omega}
\def\r{\rho}
\def\s{\sigma}

\def\th{\theta}

\def\z{\zeta}

\def\mub{{\boldsymbol{\mu}}}

\DeclareMathOperator{\diag}{{\mathrm{diag}}}

\DeclareMathOperator{\Id}{{\mathrm{Id}}}

\DeclareMathOperator{\im}{{\mathrm{Im}}}

\DeclareMathOperator{\Irr}{{\mathrm{Irr}}}
\DeclareMathOperator{\Ker}{{\mathrm{Ker}}}

\DeclareMathOperator{\Mat}{{\mathrm{Mat}}}

\DeclareMathOperator{\Res}{{\mathrm{Res}}}

\DeclareMathOperator{\Core}{{\mathrm{Core}}}

\def\to{\rightarrow}
\def\longto{\longrightarrow}
\def\injto{\hookrightarrow}

\def\fonction#1#2#3#4#5{\begin{array}{rccc}
{#1} : & {#2} & \longto & {#3}  \\
& {#4} & \longmapsto & {#5} 
\end{array}}

\def\vide{\varnothing}
\def\DS{\displaystyle}

\def\finl{~$\blacksquare$}

\def\lexp#1#2{\kern\scriptspace\vphantom{#2}^{#1}\kern-\scriptspace#2}
\def\le{\hspace{0.1em}\mathop{\leqslant}\nolimits\hspace{0.1em}}
\def\ge{\hspace{0.1em}\mathop{\geqslant}\nolimits\hspace{0.1em}}

\mathchardef\inferieur="321E
\mathchardef\superieur="321F

\def\eqna{\begin{eqnarray*}}
\def\endeqna{\end{eqnarray*}}

\def\itemth#1{\item[${\mathrm{(#1)}}$]}

\def\gfq{{\FM_{\! q}}}

\catcode`\@=11
\long\def\@car#1#2\@nil{#1}
\long\def\@first#1#2{#1}
\long\def\@second#1#2{#2}
\long\def\ifempty#1{\expandafter\ifx\@car#1@\@nil @\@empty
  \expandafter\@first\else\expandafter\@second\fi}
\catcode`\@=12

\def\GL{{\mathrm{GL}}}

\DeclareMathOperator{\Ref}{Ref}

\def\boitegrise#1#2{\begin{centerline}{\fcolorbox{black}{shadecolor}{~
    \begin{minipage}[t]{#2}{\vphantom{~}#1\vphantom{$A_{\DS{A_A}}$}}
            \end{minipage}~}}\end{centerline}\medskip}

\theoremstyle{remark}
\newtheorem{rema}[theo]{Remark}

\newtheorem{exemple}[theo]{Example}

\theoremstyle{plain}

\def\BIL{LR}
\def\GAUCHE{L}
\def\CAR{CAR}
\def\FAM{FAM}

\def\euler{{\mathrm{eu}}}

\def\xyinj{\ar@{^{(}->}}
\def\xysur{\ar@{->>}}

\def\res{{\mathrm{res}}}

\bigskip

\DeclareMathOperator{\Rep}{{\mathrm{Rep}}}

\makeatletter
\def\hlinewd#1{%
\noalign{\ifnum0=`}\fi\hrule \@height #1 %
\futurelet\reserved@a\@xhline}
\makeatother

\newlength\epaisLigne

\usepackage{multirow}

\newcommand{\longiso}{\stackrel{\sim}{\longrightarrow}}
\newcommand{\longbij}{\stackrel{\sim}{\leftrightarrow}}

\def\Rees{\operatorname{Rees}\nolimits}
\def\codim{\operatorname{codim}\nolimits}

\makeatletter
\def\hlinewd#1{%
\noalign{\ifnum0=`}\fi\hrule \@height #1 %
\futurelet\reserved@a\@xhline}
\makeatother

\usepackage{multirow}

\newcommand{\longsurto}{\relbar\joinrel\twoheadrightarrow}

\usepackage{lscape}

\usepackage{pstricks}

\def\GL{\operatorname{\Gb\Lb}\nolimits}

\def\hd{{\widehat d}}

\begin{document}

\title{Fixed points in smooth Calogero-Moser spaces}

\author{{\sc C\'edric Bonnaf\'e}}
\address{
Institut Montpelli\'erain Alexander Grothendieck (CNRS: UMR 5149), 
Universit\'e Montpellier 2,
Case Courrier 051,
Place Eug\`ene Bataillon,
34095 MONTPELLIER Cedex,
FRANCE} 

\makeatletter
\email{cedric.bonnafe@umontpellier.fr}
\makeatother

\author{{\sc Ruslan Maksimau}}

\address{
Institut Montpelli\'erain Alexander Grothendieck (CNRS: UMR 5149), 
Universit\'e Montpellier 2,
Case Courrier 051,
Place Eug\`ene Bataillon,
34095 MONTPELLIER Cedex,
FRANCE} 
\email{ruslan.maksimau@umontpellier.fr}

\date{\today}

\thanks{The first author is partly supported by the ANR 
(Project No ANR-16-CE40-0010-01 GeRepMod)}

\maketitle
\pagestyle{myheadings}
\markboth{\sc C. Bonnaf\'e \& R. Maksimau}{\sc Fixed points in smooth Calogero-Moser spaces}

\bigskip

\begin{abstract} 
We prove that every irreducible component of the fixed point variety under the action of $\mub_d$ 
in a smooth Calogero-Moser space is isomorphic to a Calogero-Moser space associated with another 
reflection group. 
\end{abstract}

Consider a finite subgroup $W\subset \GL_\CM(V)$ of automorphisms of a finite dimensional complex vector space $V$,  
generated by reflections and let $c$ be a complex valued function on the set of conjugacy classes of 
reflections of $W$. To the triple $(V,W,c)$, Etingof and Ginzburg~\cite{EG} 
have associated an affine variety $\ZCB_{\! c}(V,W)$, the {\it Calogero-Moser space}, 
which is defined as the spectrum of the center of the rational Cherednik algebra 
$\Hb_c$ {\it ``at $t=0$''} (see Section~\ref{sec:set-up} for  a precise definition). 
The algebra $\Hb_c$ carries a $\ZM$-grading (i.e., a $\CM^\times$-action), which induces 
a $\CM^\times$-action on the Calogero-Moser space $\ZCB_{\! c}(V,W)$. 

\smallskip

We define a {\it reflection subquotient} of $(V,W)$ to be a pair $(V',W')$ where 
$V'$ is a subspace of $V$, $W' \subset \Gb\Lb_\CM(V')$ is generated by reflections 
and there exists a subgroup $H$ of $W$ which stabilizes $V'$ and whose image in $\Gb\Lb_\CM(V')$ 
is exactly $W'$. 

Note that any element in $\Gb\Lb_\CM(V)$ normalizing $W$ acts on $\ZCB_{\! c}(V,W)$. 
In~\cite[Conjecture~FIX]{BR}, Rouquier and the first author proposed the following conjecture:

\bigskip

\begin{quotation}
\noindent{\bf Conjecture F.} {\it Let $\s \in \Gb\Lb_\CM(V)$ be an element of finite order and 
normalizing $W$ and let $\XCB$ 
be an irreducible component of $\ZCB_{\! c}(V,W)^\s$ (endowed with its reduced scheme structure). 
Then there exists a reflection subquotient $(V',W')$ of $W$ and a complex valued function 
$c'$ on the set of conjugacy classes of reflections in $W'$ such that 
$\XCB \simeq \ZCB_{\! c'}(V',W')$, as varieties endowed with a $\CM^\times$-action.}
\end{quotation}

\bigskip

Note that the conjecture stated in~\cite[Conjecture~FIX]{BR} does not give an explicit description 
of $c'$ in terms of $c$: it is just mentioned that the map $c \mapsto c'$ should 
be linear. Note also that, as stated, it might be a little bit optimistic: 
maybe one should replace $\XCB$ by its normalization (as it is not clear whether 
$\XCB$ is normal or not). A special case of an element $\s$ normalizing $W$ 
is when $\s$ is a root of unity, viewed as the corresponding homothety on $V$. 
The aim of this paper is to prove the following result:

\bigskip

\noindent{\bf Theorem.} {\it Conjecture~F holds if $\ZCB_{\! c}(V,W)$ 
is smooth and $\s$ is a root of unity.}

\bigskip

The proof is by case-by-case analysis, as the triples $(V,W,c)$ such that 
$\ZCB_{\! c}(V,W)$ is smooth are classified (see the works of Etingof-Ginzburg~\cite{EG}, 
Gordon~\cite{gordon},~Bellamy~\cite{bellamy}). The classification and the conjecture can be 
easily reduced to the case where $W$ acts irreducibly on $V/V^W$ and in this case, 
the smoothness of $\ZCB_{\! c}(V,W)$ implies that $W$ is of type $G(l,1,n)$ or $G_4$ in 
Shephard-Todd classification. The case of $G_4$ can be handled by computer calculations 
(see Section~\ref{sec:g4}), while the infinite family case will be handled 
by using an isomorphism between $\ZCB_{\! c}(V,W)$ and 
a quiver variety.

\bigskip

\noindent{\it Commentary.} 
The motivation for Conjecture~F comes from the modular 
representation theory of finite reductive groups and conjectures of 
Brou\'e-Malle-Michel about the endomorphism algebra of 
some Deligne-Lusztig variety (see~\cite[\S{1.A}]{brouemalle} 
and~\cite[Conj.~5.7]{brouemichel}). Experimentally, 
there is an astonishing analogy between the combinatorics 
involved in Brou\'e-Malle-Michel conjectures and the one involved 
in the geometry of $\ZC_{\! c}(V,W)^{\mub_d}$. 

The results of this paper confirm this analogy in the 
particular case of the general linear group $G=\Gb\Lb_n(\gfq)$ 
(whose Weyl group $W$ is the symmetric group $\SG_n$). Let us be more 
precise. Let $\ell$ be a prime number not dividing $q$ and let $d$ 
denote the order of $q$ modulo $\ell$. Then (see~\S{4}):
\begin{itemize}
\item The $\ell$-blocks of unipotent characters of $G$ are parametrized 
by the possible $d$-cores of partitions of $n$, as well 
as the irreducible components of $\ZC_{\! 1}(\CM^n,\SG_n)^{\mub_d}$.

\item Unipotent characters lying in the block $\BC_\g$ parametrized 
by the $d$-core $\g$ are in bijection with the irreducible 
characters of $G(d,1,r)$ (where $r=(n-|\g|)/d$). On the other side, 
the irreducible component $\XC_\g$ 
of $\ZC_{\! 1}(\CM^n,\SG_n)^{\mub_d}$ parametrized by $\g$ 
is isomorphic to $\ZC_{\! c'}(\CM^r,G(d,1,r))$ and so its 
$\CM^\times$-fixed points are also parametrized by $\Irr G(d,1,r)$. 

\item The conjectures of Brou\'e-Malle-Michel say that 
the unipotent characters lying in $\BC_\g$ are exactly the ones 
appearing in the cohomology of some Deligne-Lusztig 
variety $\Xb_\g$, and that the endomorphism algebra 
of the cohomology of $\Xb_\g$ is isomorphic to some 
Hecke algebra $\HC(G(d,1,r),c'')$ for some explicitly given 
parameter $c''$. Example~\ref{ex:parametre} show that $c''=c'$ (!).
\end{itemize}

\bigskip

\section{Notation and main result}\label{sec:set-up}

\medskip

All along this note, we will abbreviate $\otimes_\CM$ as $\otimes$. By an algebraic variety, 
we mean a reduced scheme of finite type over $\CM$. 

\bigskip

\boitegrise{{\bf Set-up.} {\it We fix in this paper a $\CM$-vector space $V$ 
of finite dimension $n$ and a finite subgroup $W$ of $\GL_\CM(V)$. We set
$$\Ref(W)=\{s \in W~|~\dim_\CM V^s=n-1\}$$
and we assume that
$$W=\langle \Ref(W) \rangle.$$
We also fix an element $\s \in \Gb\Lb_\CM(V)$ of finite order normalizing $W$.}}{0.75\textwidth}

\medskip

We set $\e : W \to \CM^\times$, $w \mapsto \det(w)$. 
If $s \in \Ref(W)$, we denote by $\a_s^\vee$ and $\a_s$ two elements of $V$ and $V^*$ respectively 
such that $V^s=\Ker(\a_s)$ and $V^{* s}=\Ker(\a_s^\vee)$, where $\a_s^\vee$ is viewed as a linear 
form on $V^*$.

\bigskip
\def\la{\langle}
\def\ra{\rangle}

\subsection{Rational Cherednik algebra at $t=0$} 
\label{sub:Cher}
All along this note, we fix a function $c : \Ref(W) \to \CM$ which 
is invariant under conjugacy. We define the $\CM$-algebra $\Hb_c$ to be the quotient 
of the algebra $\Trm(V\oplus V^*)\rtimes W$ (the semi-direct product of the tensor algebra 
$\Trm(V \oplus V^*)$ with the group $W$) 
by the relations 
$$\begin{cases}
[x,x']=[y,y']=0,\\
[y,x]=\DS{\sum_{s\in\Ref(W)}(\e(s)-1)c_s
\frac{\la y,\alpha_s\ra\la\alpha_s^\vee,x\ra}{\la\alpha_s^\vee,\alpha_s\ra}s},
\end{cases}\leqno{(\HC_c)}$$
for all $x$ ,$x'\in V^*$, $y$, $y'\in V$. The algebra $\Hb_c$ is called the 
{\it rational Cherednik algebra at $t=0$}. 

The first commutation relations imply that 
we have morphisms of algebras $\CM[V] \to \Hb_c$ and $\CM[V^*] \to \Hb_c$. 
Recall~\cite[Corollary~4.4]{EG} 
that we have an isomorphism of $\CM$-vector spaces 
\equat\label{eq:pbw}
\CM[V] \otimes \CM W \otimes \CM[V^*] \longiso \Hb_c
\endequat
induced by multiplication (this is the so-called {\it PBW-decomposition}). 

We denote by $\Zb_c$ the center of $\Hb_c$: it is well-known~\cite[Lemma~3.5]{EG} that 
$\Zb_c$ is an integral domain, which is integrally closed and contains 
$\CM[V]^W$ and $\CM[V^*]^W$ as subalgebras (so it contains $\Pb=\CM[V]^W \otimes \CM[V^*]^W$), 
and which is a free $\Pb$-module of rank $|W|$. We denote by $\ZCB_{\! c}$ the 
algebraic variety whose ring of regular functions $\CM[\ZCB_{\! c}]$ is $\Zb_c$: 
this is the {\it Calogero-Moser space} associated with the datum $(V,W,c)$. 
If necessary, we will write $\ZCB_{\! c}(V,W)$ for $\ZCB_{\! c}$.

\bigskip

\begin{exemple}\label{ex:cyclique}
Let $l \ge 1$, let $\zeta$ be a primitive $l$-th root of unity and let $\mub_l=\langle \zeta\rangle$. 
We assume in this example that $n=1$ and $W=\langle t \rangle \simeq \mub_l$, where $t(v)=\zeta v$ for all 
$v \in V$. Then $\Ref(W)=\{t,t^2,\dots,t^{l-1}\}$ and we set for 
simplification
$$k_j=\frac{1}{l} \sum_{i=1}^{l-1} \zeta^{-i(j-1)} c_{t^i}$$
($0 \le j \le l-1$). Note that $k_0+k_1+\cdots k_{l-1}=0$. Then
$$\ZCB_{\! c} \simeq \{(x,y,e) \in \CM^3~|~\prod_{i=0}^{l-1} (e-lk_i)=xy\}$$
(see for instance~\cite[Theorem~18.2.4]{BR}). In particular, $\ZCB_{\!c}$ is smooth if and only if 
$$\prod_{0 \le i < j \le l-1} (k_i-k_j) \neq 0.$$
Note also that the inclusion $\Pb \subset \Zb_c$ corresponds to the morphism 
of algebraic varieties $\ZCB_{\! c} \to \CM^2$, $(x,y,e) \mapsto (x,y)$.

The $\CM^\times$-action is given by the formula 
$$\xi \cdot (x,y,e)=(\xi^d x,\xi^{-d}y,e),$$
where $\xi \in \CM^\times$ and $(x,y,e) \in \ZCB_{\! c}$.\finl
\end{exemple}

\bigskip

\subsection{Action of the normalizer}
The normalizer $N_{\Gb\Lb_\CM(V)}(W)$ of $W$ in $\Gb\Lb_\CM(V)$ 
acts on $V$ naturally, on $V^*$ by the contragredient action 
and on $W$ by conjugation. This endows $\Trm(V \oplus V^*) \rtimes W$ with an action 
of $N_{\Gb\Lb_\CM(V)}(W)$ and it is easily checked that the bunch of relations~$(\HC_c)$ is stable 
under this action. So $\Hb_c$ inherits an action of $N_{\Gb\Lb_\CM(V)}(W)$. In particular, 
its centre $\Zb_c$ also inherits such an action, and this defines an action of 
$N_{\Gb\Lb_\CM(V)}(W)$ on the Calogero-Moser space $\ZCB_{\! c}$.

\bigskip

\begin{theo}\label{theo:main}

Assume that $\ZCB_{\! c}$ is smooth and that $\s$ is a root of unity, and let 
$\XCB$ be an irreducible component of $\ZCB_{\! c}^\s$. Then there exists 
a reflection subquotient $(V',W')$ of $(V,W)$ and a complex-valued map $c'$ on the 
set of conjugacy classes of reflections of $W'$ such that 
$$\XCB \simeq \ZCB_{\! c'}(V',W'),$$
as varieties endowed with a $\CM^\times$-action.
\end{theo}

\bigskip

\begin{rema}\label{rem:stupide}
Note that if $\XCB$ is reduced to a point, then the theorem is easy because 
in this case, $\XCB \simeq \ZCB_{\! 0}(0,1)$.\finl
\end{rema}

\bigskip

This Theorem will be proved in Sections~\ref{sec:reductions} and~\ref{sec:g4}. 

\bigskip

\def\cod{{\mathrm{cod}}}

\subsection{Filtration of the group algebra}
If $w \in W$, we set $\cod(w)=\codim_\CM(V^w)$. We define a filtration $\FC_\bullet(\CM W)$ 
of the group algebra of $W$ as follows: let
$$\FC_i(\CM W)=\bigoplus_{\cod(w) \le i} \CM w.$$
Then 
$$\CM\Id_V=\FC_0(\CM W)\subset \FC_1(\CM W)\subset \cdots\subset\FC_n(\CM W)=\CM W=\FC_{n+1}(\CM W)= \cdots$$ 
is a filtration of $\CM W$. For any subalgebra $A$ of $\CM W$, we set $\FC_i(A)=A\cap\FC_i(\CM W)$, 
so that 
$$\CM\Id_V=\FC_0(A)\subset \FC_1(A)\subset \cdots\subset\FC_n(A)=A=\FC_{n+1}(A)=\cdots$$ 
is also a filtration of $A$. 

Now, assume that $\ZCB_{\! c}(V,W)$ is smooth, let $\s$ be a root of unity and let 
$\XC$ be an irreducible component of $\ZCB_{\! c}(V,W)^\s$. According to Theorem~\ref{theo:main}, 
there exists a reflection subquotient $(V',W')$ of $(V,W)$, a complex-valued map $c'$ on the 
set of conjugacy classes of reflections of $W'$ and a $\CM^\times$-equivariant 
isomorphism $i : \ZCB_{\! c'}(V',W') \longiso \XC$. We will view $i$ as a closed 
immersion $i : \ZCB_{\! c'}(V',W') \injto \ZCB_{\! c}(V,W)$. 
By~\cite[Corollary~5.8]{gordon}, the smoothness of $\ZCB_{\! c}(V,W)$ implies that there 
exists a bijection between $\Irr(W)$ and $\ZCB_{\! c}(V,W)^{\CM^\times}$. 
As the fixed point variety of a finite group in a smooth complex algebraic variety 
is still smooth, this means that $\XCB$ is smooth, and so $\ZCB_{\! c'}(V',W')$ is smooth. 
Applying again~\cite[Corollary~5.8]{gordon}, we get another bijection between 
$\Irr(W')$ and $\ZCB_{\! c'}(V',W')^{\CM^\times}$. As $\XC^{\CM^\times} \subset \ZCB_{\! c}(V,W)^{\CM^\times}$, 
this gives an injective map $i_\XC : \Irr(W') \injto \Irr(W)$, depending on $\XC$ and the choice 
of $i$. This allows to define a surjective morphism of algebras 
$$i_\XC^* : \Zrm(\CM W) \longsurto \Zrm(\CM W')$$
as follows: if $\chi \in \Irr(W)$, let $e_\chi^W$ denote the corresponding primitive 
central idempotent of $\CM W$ and set
$$i_\XC^*(e_\chi^W)=
\begin{cases}
e_{\chi'}^{W'} & \text{if $\chi' \in \Irr(W')$ is such that $i_\XC(\chi')=\chi$,}\\
0 & \text{if $\chi$ does not belong to the image of $i_\XC$.}
\end{cases}
$$
Using results of Shan and the first author~\cite{BS}, Theorem~\ref{theo:main} has the following consequence:

\bigskip

\begin{coro}\label{coro:filtration}
If $i \ge 0$, then $i_\XC^*(\FC_i(\Zrm(\CM W))) \subset \FC_i(\Zrm(\CM W'))$. 
\end{coro}

\bigskip
\def\pt{{\mathbf{pt}}}

\begin{proof}
If $\YCB$ is a complex algebraic variety endowed with a $\CM^\times$-action, 
we denote by $\Hrm^i_{\CM^\times}(\YCB)$ the $i$-th group of equivariant cohomology, 
with coefficients in $\CM$. 
Let $\hbar$ be an indeterminate and identify $\Hrm_{\CM^\times}(\pt)$ with $\CM[\hbar]$ 
as usual, with $\hbar$ homogeneous of degree $2$. Since $\ZCB_{\! c}(V,W)$ is smooth, it follows from~\cite[Theorem~A]{BS} 
that 
$$
\begin{cases}
\Hrm^{2i+1}_{\CM^\times}(\ZCB_{\! c}(V,W)) = 0 & \text{if $i \ge 0$,}\\
\Hrm_{\CM^\times}^{2\bullet} \simeq \Rees_\FC(\Zrm(\CM W)) & \text{as $\CM[\hbar]$-algebras,}\\
\end{cases}
$$
where $\Rees_\FC(\Zrm(\CM W)) = \bigoplus_{i \ge 0} \hbar^i \FC_i(\Zrm(\CM W)) \subset \CM[\hbar] \otimes_\CM \Zrm(\CM W)$ 
is the Rees algebra associated with the filtration $\FC_\bullet(\Zrm(\CM W))$. Using the analogous result 
for $\ZCB_{\! c'}(V',W')$ and the functoriality of equivariant cohomology, we get a morphism 
of algebras $i^* : \Rees_\FC(\Zrm(\CM W)) \longto \Rees_\FC(\Zrm(\CM W'))$. 

Now, we will use not only~\cite[Theorem~A]{BS} but also its proof: the proof goes by restriction 
to the fixed point subvariety $\ZCB_{\! c}(V,W)^{\CM^\times}$ and identifying 
$\Hrm^{2\bullet}_{\CM^\times}(\ZCB_{\! c}(V,W)^{\CM^\times})$ with $\CM[\hbar] \otimes \Zrm(\CM W)$ 
using~\cite[Corollary~5.8]{gordon}. As the same strategy holds for $\ZCB_{\! c'}(V',W')$, 
the functoriality of equivariant cohomology implies that the map $i^*$ fits into a commutative diagram
$$\diagram
\Rees_\FC(\Zrm(\CM W)) \xto[0,3]^{\DS{i^*}} \ar@{^(->}[dd] &&& \Rees_\FC(\Zrm(\CM W')) \ar@{^(->}[dd]\\
&& \\
\CM[\hbar] \otimes_\CM \Zrm(\CM W) \xto[0,3]^{\DS{\Id_{\CM[\hbar]} \otimes_\CM i_\XC^*}} &&& \CM[\hbar] \otimes_\CM \Zrm(\CM W').
\enddiagram$$
This shows the corollary.
\end{proof}

\bigskip

In Remark~\ref{rem:W'-W} and Theorem \ref{theo:recapitulatif}, we will describe combinatorially the injective 
map $i_\XC : \Irr(W') \injto \Irr(W)$ whenever $W$ is of type $G(l,1,n)$.

\bigskip

\section{Preliminaries on quiver varieties}\label{sec:quiver}

\medskip

\subsection{Quiver varieties}\label{sub:quiver}
Let $Q_l$ denote the cyclic quiver with $l$ vertices, defined as follows:
\begin{itemize}
\item[$\bullet$] Vertices: $i \in \ZM/l\ZM$;

\item[$\bullet$] Arrows: $y_i : i \longto i+1$, $i \in \ZM/l\ZM$. 
\end{itemize}
We denote by $\Qov_l$ the double quiver of $Q_l$ that is, the quiver obtained from $Q_l$ by adding an arrow 
$x_i : i+1 \to i$ for all $i \in \ZM/l\ZM$ (see Figure 1).

\begin{centerline}{
\psset{dimen=middle,xunit=1cm, yunit=1cm}
\begin{pspicture}(-3,-3)(3,3)%
\psset{linecolor=[rgb]{0,0,0},linewidth=0.4pt,linestyle=dashed,dash=2.5pt 2pt,linejoin=1}
\psellipticarcn(0,0)(2,2){77.1429}{231.4286}
\psset{linestyle=solid,fillstyle=solid,fillcolor=[rgb]{1,1,1}}
\rput{90}(-2,0){\psellipse(0,0)(0.1,0.1)}
\rput{90}(-1.247,-1.5637){\psellipse(0,0)(0.1,0.1)}
\rput{90}(-1.247,1.5637){\psellipse(0,0)(0.1,0.1)}
\rput{90}(0.445,1.9499){\psellipse(0,0)(0.1,0.1)}
\rput(-2.3,-0){\small $0$}
\rput(-1.434,1.7982){\small $1$}
\rput(0.5118,2.2423){\small $2$}
\rput(-1.434,-1.7982){\small $-1$}
\psset{linewidth=1.2pt,fillstyle=none}
\psellipticarcn{->}(0,0)(2,2){225.4286}{186}
\psellipticarcn{->}(0,0)(2,2){174}{134.5714}
\psellipticarcn{->}(0,0)(2,2){122.5714}{83.1429}
\psellipticarcn{->}(-3.247,-1.5637)(2,2){45.4286}{6}
\psellipticarcn{->}(-3.247,1.5637)(2,2){354}{314.5714}
\psellipticarcn{->}(-0.8019,3.5135)(2,2){302.5714}{263.1429}
\psset{linewidth=0.4pt}
\rput(-2.0722,-0.9979){\small $y_{-1}$}
\rput(-2.0722,0.9979){\small $y_0$}
\rput(-0.5118,2.2423){\small $y_1$}
\rput(-0.2893,1.2674){\small $x_1$}
\rput(-1.1713,0.564){\small $x_0$}
\rput(-1.1713,-0.564){\small $x_{-1}$}
\end{pspicture}%
}\end{centerline}

\vskip-0.8cm

\begin{centerline}{{\bf Fig. 1.} The quiver $\Qov_l$}\end{centerline}

\medskip

Now, if $d=(d_i)_{i \in \ZM/l\ZM}$ is a family of elements of $\NM$, 
we denote by $\Gb\Lb(d)$ the direct product
$$\Gb\Lb(d)=\prod_{i \in \ZM/l\ZM} \Gb\Lb_{d_i}(\CM),$$
by $\D\CM^\times$ the image of $\CM^\times$ in $\Gb\Lb(d)$ through the 
diagonal embedding $\D : \CM^\times \injto \Gb\Lb(d)$ and 
we set
$$\Pb\Gb\Lb(d)=\Gb\Lb(d)/\D\CM^\times.$$
The group $\Pb\Gb\Lb(d)$ acts on the variety $\Rep(\Qov_l,d)$ of representations 
of $\Qov_l$ in the family of vector spaces $(\CM^{d_i})_{i \in \ZM/l\ZM}$. The orbits 
are the isomorphism classes of representations of $\Qov_l$ of dimension vector $d$. 
We denote by 
$$\fonction{\mu_d}{\Rep(\Qov_l,d)}{\bigoplus_{i \in \ZM/l\ZM} 
\Mat_{d_i}(\CM)}{(X_i,Y_i)_{i \in \ZM/l\ZM}}{(X_iY_i - Y_{i-1}X_{i-1})_{i \in \ZM/l\ZM}}$$
the corresponding {\it moment map}. Finally, if $\th=(\th_i)_{i \in \ZM/l\ZM}$ is a family of complex numbers, 
we denote by $\Irm_\th(d)$ the family $(\th_i \Id_{\CM^{d_i}})_{i \in \ZM/l\ZM}$ and by $\OC_\th(d)$ 
the closed subvariety of $\Mat_{d_0}(\CM)$ consisting of matrices of rank $\le 1$ and trace $-\sum_{i\in \ZM/l\ZM} \th_i d_i$. 
Finally, we set
$$\YC_\th(d)=\mu_d^{-1}(\Irm_\th(d) + \OC_\th(d))\qquad\text{and}\qquad
\XC_\th(d)=\YC_\th(d)/\!\!/\Pb\Gb\Lb(d).$$
We will usually denote an element $(X_i,Y_i)_{i \in \ZM/l\ZM}$ of $\YC_\th(d)$ by $(X,Y)$, where $X=(X_i)_{i \in \ZM/l\ZM}$ 
and $Y=(Y_i)_{i \in \ZM/l\ZM}$. 
Note that $\YC_\th(d)$ is endowed with a $\CM^\times$-action: if $\xi \in \CM^\times$, we set 
$$\xi \cdot (X,Y)=(\xi^{-1}X,\xi Y).$$
This action commutes with the action of $\Gb\Lb(d)$ and the moment map is constant on 
$\CM^\times$-orbits, so it induces a $\CM^\times$-action on $\XC_\th(d)$. 

\bigskip

\begin{rema}\label{rem:dim-negative}
We extend the definition of $\XC_\th(d)$ to the case where $d \in \ZM^{\ZM/l\ZM}$ 
by the convention that $\XC_\th(d)=\vide$ whenever at least one of the $d_i$'s is negative.\finl
\end{rema}

\def\aff{\mathrm{aff}}

\bigskip

\subsection{Action of the affine Weyl group}\label{sub:action-weyl}
If $l \ge 2$, let $W_l^\aff$ denote the affine Weyl group of type $\Ati_{l-1}$. It is the 
Coxeter group with associated Coxeter system $(W_l^\aff,S_l^\aff)$, where 
$S_l^\aff=\{s_i~|~i \in \ZM/l\ZM\}$ and the Coxeter graph is given by

\begin{centerline}{
\psset{dimen=middle,xunit=1cm, yunit=1cm}
\begin{pspicture}(-3,-3)(3,3)%
\psset{linecolor=[rgb]{0,0,0},linewidth=0.4pt,linejoin=1}
\psline(0.445,1.9499)(-1.247,1.5637)(-2,-0)(-1.247,-1.5637)
\psset{linestyle=dashed,dash=2.5pt 2pt}
\psellipticarcn(0,0)(2,2){77.1429}{231.4286}
\psset{linestyle=solid,fillstyle=solid,fillcolor=[rgb]{1,1,1}}
\rput{90}(-2,0){\psellipse(0,0)(0.1,0.1)}
\rput{90}(-1.247,-1.5637){\psellipse(0,0)(0.1,0.1)}
\rput{90}(-1.247,1.5637){\psellipse(0,0)(0.1,0.1)}
\rput{90}(0.445,1.9499){\psellipse(0,0)(0.1,0.1)}
\rput(-2.3,-0){\small $s_0$}
\rput(-1.434,1.7982){\small $s_1$}
\rput(0.5118,2.2423){\small $s_2$}
\rput(-1.434,-1.7982){\small $s_{-1}$}
\end{pspicture}%
}\end{centerline}

\vskip-0.8cm

\begin{centerline}{{\bf Fig. 2.} Coxeter graph of $(W_l^\aff,S_l^\aff)$}\end{centerline}

\medskip


We extend this notation to the case where $l=1$ by setting $W_1^\aff=1$.

Consider the Lie algebra $\gG_l=\mathfrak{sl}_l(\CM)$ and its affine version  
$\widehat{\gG}_l=\widehat{\mathfrak{sl}}_l(\CM)=\mathfrak{sl}_l(\CM)[t,t^{-1}]\oplus \CM \mathbf{1}\oplus \CM\partial$.
Let $\hG\subset \gG_l$ be the Cartan subalgebra formed by the diagonal 
matrices and set $\widehat{\hG}=\hG\oplus \CM \mathbf{1}\oplus \CM\partial$. 

The $\CM$-vector space $\widehat{\hG}^*$ has a basis $(\alpha_0,~ \alpha_1,\ldots,~\alpha_{l-1},\Lambda_0)$, 
where $\alpha_0,~ \alpha_1,\ldots,~\alpha_{l-1}$ are the simple roots of $\widehat{\gG}_l$ and $\Lambda_0$ 
is such that $\Lambda_0$ annihilates $\hG$ and $\partial$ and $\Lambda_0(\mathbf{1})=1$. Denote by $R^\aff_l$ 
and $R_l$ the affine and the non-affine root 
lattices respectively (i.e., $R_l^\aff$ is the $\ZM$-lattice generated by 
$\alpha_0,~ \alpha_1,\ldots,~\alpha_{l-1}$ and $R_l$ is the sublattice generated 
by $\alpha_1,\ldots,~\alpha_{l-1}$.)

Following~\cite{lusztig}, 
we define two actions of $W_l^\aff$: a non-linear one on $\ZM^{\ZM/l\ZM}$, 
and a linear one on $\CM^{\ZM/l\ZM}$. If $l=1$, there is nothing to define so we may assume that 
$l \ge 2$. If $d=(d_i)_{i \in \ZM/l\ZM} \in \ZM^{\ZM/l\ZM}$ and if 
$j \in \ZM/l\ZM$, we set $s_j(d)=(d_i')_{i \in \ZM/l\ZM}$, where 
$$d_i'=
\begin{cases}
d_i & \text{if $i \neq j$,}\\
\d_{j0} + d_{i+1}+d_{i-1}-d_i & \text{if $i=j$.}\\
\end{cases}
$$

\bigskip
\begin{rema}
\label{rem:twisted_action}
We can identify $\ZM^{\ZM/l\ZM}$ with the root lattice $R_l^\aff$ by 
$$
d\mapsto \sum_{i\in \ZM/l\ZM}d_i\alpha_i.
$$
Beware, the action considered here is not the usual action of $W_l^\aff$ on the root lattice. 
When we have $w(d)=d'$ with respect to the action defined above, this corresponds to $w(\Lambda_0-d)=\Lambda_0-d'$ for 
the usual action of $W_l^\aff$ on $\widehat \hG^*$.\finl
\end{rema}

If $\th=(\th_i)_{i \in \ZM/l\ZM} \in \CM^{\ZM/l\ZM}$, we set $s_j(\th)=(\th_i')_{i \in \ZM/l\ZM}$, where
$$\th_i'=
\begin{cases}
\th_i & \text{if $i \not\in \{j-1,j,j+1\}$,}\\
\th_j+\th_i & \text{if $i \in \{j-1,j+1\}$,}\\
-\th_i & \text{if $i=j$.}
\end{cases}
$$
It is readily seen that these definitions on generators extend to an action 
of the whole group $W_l^\aff$. 
We also define a pairing $\ZM^{\ZM/l\ZM} \times \CM^{\ZM/l\ZM} \to \CM$, $(d,\th) \mapsto d \cdot \th$, 
where 
$$d \cdot \th = \sum_{i \in \ZM/l\ZM} d_i \th_i.$$
Then
\equat\label{eq:pairing}
s_j(d) \cdot s_j(\th) = (d \cdot \th) - \d_{j0} \th_0.
\endequat
It is proved in~\cite[Corollary~3.6]{lusztig} that
\equat\label{eq:iso-lusztig}
\XC_{s_j(\th)}(s_j(d)) \simeq \XC_\th(d)\qquad \mbox{ if }\theta_j\ne 0.
\endequat
Note that this isomorphism takes into account the convention of Remark~\ref{rem:dim-negative}. 



The isomorphism above motivates to consider the following equivalence relation on the set $\ZM^{\ZM/l\ZM} \times \CM^{\ZM/l\ZM}$. Let $\sim$ be the transitive closure of 
$$
(d,\theta)\sim (s_i(d),s_i(\theta)), \qquad \theta_i\ne 0.
$$
The isomorphism \eqref{eq:iso-lusztig} implies that if $(d,\theta)\sim(d',\theta')$, then we have an isomorphism of algebraic varieties $\XC_\theta(d)\simeq \XC_{\theta'}(d')$.

Recall that the affine Weyl group has another presentation. We have $W_l^\aff=W_l\ltimes R_l$, where $W_l=\mathfrak{S}_l$ 
is the non-affine Weyl group, that is, the subgroup generated by $s_1$,\dots, $s_{l-1}$. 
For each $\alpha\in R_l$, denote by $t_\alpha$ the image of $\alpha$ in $W^\aff_l$. 
Each element of $W_l^\aff$ can be written in a unique way in the form $w\cdot t_\alpha$, where $w\in W_l$ and $\alpha\in R_l$. 
Let $\d_l$ denote the constant family $\d_l=(1)_{i \in \ZM/l\ZM} \in \ZM^{\ZM/l\ZM}$. 

\bigskip

Recall also from~\cite[\S{11}]{EG} the following result, which follows from 
the fact that $\XC_\th(n\d_l)$ is isomorphic to some Calogero-Moser space 
(we will use this fact later, and make this statement more precise):

\bigskip

\begin{lem}\label{lem:dim}
If $n \ge 0$, then $\XC_\th(n\d_l)$ is normal and irreducible of dimension $2n$.
\end{lem}

\bigskip
\begin{defi}
\label{def:smooth-pair}
We say that the pair $(d,\theta)$ is smooth if it is equivalent to a pair of the form $(n\delta_l,\theta')$ such that 
$n\geqslant 0$ and $\XC_{\theta'}(n\d_l)$ is smooth. In particular, in this case, the variety $\XC_{\theta}(d)$ is 
smooth and non-empty.
\end{defi}

\bigskip

In the following lemma we identify $\ZM^{\ZM/l\ZM}$ with $R_l^\aff$.

\bigskip

\begin{lem}
\label{lem:affWeyl_translation-roots}
Assume $\alpha\in R_l$ and $d\in \ZM^{\ZM/l\ZM}$. Then we have $t_\alpha(d) \equiv d-\alpha \mod \ZM\delta_l$.
\end{lem}

\bigskip
\begin{proof}
This statement is a partial case of \cite[(6.5.2)]{Kac} (see also Remark \ref{rem:twisted_action}).
\end{proof}

\bigskip

\begin{lem}\label{lem:orbites}
For each $d\in \ZM^{\ZM/l\ZM}$, there exists a unique $n \in \ZM$ such that $d$ and $n\d_l$ are 
in the same $W_l^\aff$-orbit.
\end{lem}

\bigskip
\begin{proof}
Let us prove the existence of $n$. Set $\alpha=d-d_0\delta_l$. Then $\alpha$ is clearly an element of $R_l$. 
Then by Lemma \ref{lem:affWeyl_translation-roots}, the element $t_\alpha(d)$ is of the form $n\delta_l$.

Let us prove the uniqueness of $n$. Assume that $n_1\d_l$ and $n_2\d_l$ are in the 
same $W_l^\aff$-orbit. Since $n_1\d_l$ is $W_l$-stable, the $W_l^\aff$-orbit of $n_1\d_l$ 
coincides with the $R_l$-orbit of $n_1\d_l$. So there exists $\a \in R_l$ such that $t_\alpha(n_1\d_l)=n_2\d_l$. 
By Lemma \ref{lem:affWeyl_translation-roots}, this forces $\alpha=0$ and $n_1=n_2$.
\end{proof}

\bigskip

Consider the $\ZM$-linear map
$$
R^\aff_l\to \CM^{\ZM/l\ZM}, \quad d\mapsto \overline d,
$$
given by
$$
(\overline{\alpha_r})_i=2\delta_{i,r}-\delta_{i,r+1}-\delta_{i,r-1}.
$$
The kernel of this map is $\ZM\delta_l$. Set 
$$
\Sigma(\th)=\sum_{i \in \ZM/l\ZM} \th_i.
$$

\bigskip
\begin{lem}
\label{lem:affWeyl_translation-theta}
For each $\alpha\in R_l$ and $\theta\in \CM^{\ZM/l\ZM}$, we have $t_\alpha(\theta)=\theta+\Sigma(\th)\overline\alpha$.
\end{lem}
\bigskip
\begin{proof}
The $W^\aff_l$-action on $\CM^{\ZM/l\ZM}$ defined above coincides with the (usual) action of $W^\aff_l$ on the dual of the span of $\alpha_0,\alpha_1,\ldots,\alpha_{l-1}$ in $\widehat\hG^*$. The statement follows from \cite[(6.5.2)]{Kac}. 
\end{proof}

\bigskip

We denote by $(\ZM^{\ZM/l\ZM})_+$ the set of $d \in \ZM^{\ZM/l\ZM}$ which belong to the orbit of some $n\d_l$, with $n \ge 0$. This set has a more precise combinatorial description. For example, it will follow from Proposition \ref{prop:core=orbit0} 
that an element $d\in \ZM^{\ZM/l\ZM}$ is in $(\ZM^{\ZM/l\ZM})_+$ if and only if $d$ is a residue of some Young diagram. In particular, the set $(\ZM^{\ZM/l\ZM})_+$ is contained in $\NM^{\ZM/l\ZM}$. 

\bigskip

\begin{lem}\label{lem:descrX-if-simple}
Assume that $d \in \ZM^{\ZM/l\ZM}$ is such that there exists a simple representation in $\YC_\theta(d)$. Then 

\begin{itemize}
\item[(a)] the variety $\XC_\th(d)$ is normal and irreducible,

\item[(b)] we have $d\in (\ZM^{\ZM/l\ZM})_+$.
\end{itemize}
\end{lem}
\begin{proof}
Let $n\in \ZM$ and $w\in W^\aff_l$ be such that $n\d_l=w(d)$. Let us show that \eqref{eq:iso-lusztig} implies an 
isomorphism $\XC_\th(d)\simeq \XC_{w(\th)}(n\d_l)$. (In particular, we must have $n\geqslant 0$.) The only 
difficulty that we have is caused by the condition $\th_i\ne 0$ in \eqref{eq:iso-lusztig}. 
But \cite[Lemma~7.2]{CB} implies that we can find a sequence of reflections such that $s_{i_1}s_{i_2}\dots s_{i_r}(d)=n\d_l$ and 
such that we can apply a sequence of isomorphisms \eqref{eq:iso-lusztig} with no danger to get $\th_{i_j}=0$ while applying $s_{i_j}$.
\end{proof}

\bigskip

\subsection{Smoothness}\label{sub:smooth}
The following result is proved in~\cite[Lemma~4.3 and its proof]{gordon quiver} (which is based 
on~\cite[Theorem~1.2]{lebruyn} and~\cite[Theorem~1.2]{CB}):

\bigskip

\begin{theo}\label{theo:smooth}
Let $n \ge 0$. Then the following are equivalent:
\begin{itemize}
\itemth{1} The variety $\XC_\th(n\d_l)$ is smooth.

\itemth{2} Every element of $\YC_\th(n\d_l)$ defines a simple representation 
of the quiver $\Qov_l$.

\itemth{3} The family $\th$ satisfies
$$\Sigma(\th) \prod_{\substack{1 \le i \le j \le l-1 \\ -(n-1) \le k \le n-1}} \bigl((\th_i + \th_{i+1} + \cdots + \th_j)+k\Sigma(\th)\bigr) 
\neq 0.$$
\end{itemize}
If these equivalent conditions hold, then $\Pb\Gb\Lb(n\d_l)$ acts freely on $\YC_\th(n\d_l)$, and so, as 
a set, $\XC_\th(n\d_l)$ identifies with $\Gb\Lb(n\d_l)$-orbits in $\YC_\th(n\d_l)$. 
\end{theo}

\bigskip

\subsection{Fixed points}\label{sub:fixed}
For every $m \ge 1$, the group of $m$-th roots of unity $\mub_m$ acts on $\XC_\th(n\d_l)$ 
(this is the restriction of the $\CM^\times$-action defined in~\S\ref{sub:quiver}). 
We aim to compute the fixed point variety $\XC_\th(n\d_l)^{\mub_m}$ whenever 
$\XC_\th(n\d_l)$ is smooth. Note first that $\mub_l$ acts trivially on $\XC_\th(n\d_l)$: 
indeed, if $\z \in \mub_l$ and if $(X,Y) \in \YC_\th(n\d_l)$, 
then $\z \cdot (X,Y) = \lexp{g_0}{(X,Y)}$, where $g_0=(\z^i \Id_{\CM^n})_{i \in \ZM/l\ZM} \in \Gb\Lb(n\d_l)$. 
This means that, in order to describe $\XC_\th(n\d_l)^{\mub_m}$, we may replace $\mub_m$ 
by $\langle \mub_m,\mub_l \rangle$ or, in other words, we may assume that 
$l$ divides $m$.

\medskip

\boitegrise{{\bf Hypothesis.} {\it We fix in this subsection a non-zero natural number 
$k$ and we set $m=kl$.}}{0.75\textwidth}

\medskip

If $\th=(\th_i)_{i \in \ZM/l\ZM} \in \CM^{\ZM/l\ZM}$, we define $\th[k]$ to be the element $(\th[k]_j)_{j \in \ZM/m\ZM}$ 
such that
$$\th[k]_j = \th_i$$
if $j \equiv i \mod l$. Roughly speaking, $\th[k]$ is the concatenation of $k$ copies 
of $\th$. 

We denote by $\ZM^{\ZM/m\ZM}[n\delta_l]$ the set of $d=(d_j)_{j \in \ZM/m\ZM} \in \ZM^{\ZM/m\ZM}$ 
such that 
$$n=\sum_{\substack{j \in \ZM/m\ZM \\ j \equiv i \mod l}} d_j$$
for all $i \in \ZM/l\ZM$. Set also $\EC(k,l,n)=\ZM^{\ZM/m\ZM}[n\delta_l]\cap (\ZM^{\ZM/m\ZM})_+$.  If $d \in \EC(k,l,n)$ and 
$(X,Y) \in \Rep(\Qov_m,d)$, 
we set 
$$i_k^{(d)}(X,Y)=(X',Y'),$$
where $X'_i = \bigoplus_{\substack{j \in \ZM/m\ZM \\ j \equiv i \mod l}} X_j$ 
and $Y_i' = \bigoplus_{\substack{j \in \ZM/m\ZM \\ j \equiv i \mod l}} Y_j$. 
By the definition of $\EC(k,l,n)$, we have $X_i'$, $Y_i' \in \Mat_n(\CM)$. 
In other words, $(X',Y') \in \Rep(\Qov_l,n\d_l)$: it is clear that 
$i_k^{(d)} : \Rep(\Qov_m,d) \injto \Rep(\Qov_l,n\d_l)$ is a closed immersion. 

\bigskip

\begin{lem}\label{lem:closed-immersion}
If $d \in \EC(k,l,n)$ and $(X,Y) \in \YC_{\th[k]}(d)$, then $i_k^{(d)}(X,Y) \in \YC_\th(n\d_l)$. 
\end{lem}

\bigskip

\begin{proof}
Write $i_k^{(d)}(X,Y)=(X',Y')$ and let $i \in \ZM/l\ZM$. Write $p_0$ for the 
rank one matrix $X_0Y_0-Y_{-1}X_{-1}-\th[k]_0 \Id_{\CM^{d_0}}$. 
Then 
\eqna
X_i'Y_i'-Y_{i-1}'X_{i-1}'&=&
\DS{\bigoplus_{\substack{j \in \ZM/m\ZM \\ j \equiv i \mod l}} (X_jY_j-Y_{j-1}X_{j-1})} \\
&=& \DS{\bigoplus_{\substack{j \in \ZM/m\ZM \\ j \equiv i \mod l}} \th[k]_j \Id_{\CM^{d_j}} } + 
\d_{0i} p_0\\
&=& \th_i \Id_{\CM^n} + \d_{0i} p_0,\\
\endeqna
as desired.
\end{proof}

\bigskip

It is also clear that, if $(X,Y)$ and $(\widetilde X,\widetilde Y)$ are two elements of $\YC_{\th[k]}(d)$ which 
are conjugate under $\Gb\Lb(d)$, then $i_k^{(d)}(X,Y)$ and $i_k^{(d)}(\widetilde X,\widetilde Y)$ are conjugate 
under $\Gb\Lb(n\d_l)$. This means that $i_k^{(d)}$ induces a morphism of algebraic varieties
still denoted by
$$i_k^{(d)} : \XC_{\th[k]}(d) \longto \XC_\th(n\d_l).$$
For $d\in \EC(k,l,n)$, set 
$$
\XC_d=i_k^{(d)}(\XC_{\th[k]}(d)) \subset \XC_\th(n\d_l).
$$ 

For the moment, we allow the possibility that the varieties $\XC_{\th[k]}(d)$ and $\XC_d$ are empty. Set 
$$
\EC(k,l,n)_{\ne\vide}=\{d \in \EC(k,l,n);~\XC_d\ne \vide\}.
$$ 
We will show in Corollary \ref{coro:always-not-empty} that we have $\EC(k,l,n)_{\ne\vide}=\EC(k,l,n)$.

The next result describes the fixed point variety $\XC_\th(n\d_l)^{\mub_m}$: it 
is the main step towards the proof of Theorem~\ref{theo:main}.

\bigskip

\begin{theo}\label{theo:fixed-quiver}
Assume that $\XC_\th(n\d_l)$ is smooth. Then:
\begin{itemize}
\itemth{a} $\XC_\th(n\d_l)^{\mub_m}$ is smooth.


\itemth{b} We have 
$$
\XC_\th(n\d_l)^{\mub_m}=\coprod_{d\in \EC(k,l,n)_{\ne\vide}} \XC_d
$$
and $\XC_d$ are exactly the irreducible components of $\XC_\th(n\d_l)^{\mub_m}$.

\itemth{c} If $d \in \EC(k,l,n)_{\ne\vide}$, then $i_k^{(d)}\colon \XC_{\th[k]}(d) \to \XC_d$ is an isomorphism.

\end{itemize}
\end{theo}

\bigskip

\begin{proof}
(a) follows from the fact already mentioned in the proof of Corollary~\ref{coro:filtration} 
that the fixed point variety under the action of a finite group on a smooth variety 
is again smooth.

\medskip

We will now prove (b). Let $\xi$ be a primitive $m$-th root of unity and 
let $\z=\xi^k$ (it is a primitive $l$-th root of unity). Let $g_0$ denote the element $(\z^i \Id_{\CM^n})_{i \in \ZM/l\ZM}$ 
in $\Gb\Lb(n\d_l)$. This element satisfies 
$$\lexp{g_0}{(X,Y)}=(\z^{-1}X,\z Y)\leqno{(*)}$$
for all $(X,Y) \in \YC_\th(n\d_l)$. 
We will use the following variety:
$$\ZC=\{(X,Y,g) \in \YC_{\th}(n\d_l) \times \Gb\Lb(n\d_l)~|~
\lexp{g}{(X,Y)}=(\xi^{-1}X,\xi Y)\text{~and~}g p_{X,Y}=p_{X,Y}\},$$
where $p_{X,Y}=([X_i,Y_i]-\th_i)_{i \in \ZM/l\ZM}$: it is $0$ except on the $0$-th component, 
where it is equal to a rank one matrix of trace $-n\Sigma(\th)$ (recall from Theorem~\ref{theo:smooth} 
that $\Sigma(\th) \neq 0$ because $\XC_\th(n\d_l)$ is assumed to be smooth). 

We denote by $p' : \ZC \to \YC_\th(n\d_l)$ the projection on the first two factors, 
by $p : \ZC \to \XC_\th(n\d_l)$ the map induced by $p'$ and by $q : \ZC \to \Gb\Lb(n\d_l)$ the projection on 
the third factor. Note that $\Pb\Gb\Lb(n\d_l)$ acts on $\ZC$ by conjugacy 
on the three factors, and so the maps $p'$ and $q$ are equivariant for this action.

First, it is clear that $p(\ZC) \subset \XC_\th(n\d_l)^{\mub_m}$. 
Conversely, if $(X,Y) \in \YC_\th(n\d_l)$ is a representative of 
an element of $\XC_\th(n\d_l)^{\mub_m}$, then there exists $g \in \Pb\Gb\Lb(n\d_l)$ such that 
$\lexp{g}{(X,Y)}=(\xi^{-1}X,\xi Y)$. In particular, $\lexp{g}{p_{X,Y}}=p_{X,Y}$. 
So $g$ stabilizes the image of $p_{X,Y}$, which is of dimension $1$: this means that 
$g$ acts by a non-zero scalar $\o$ on $\im(p_{X,Y})$. Setting $g'=\o^{-1} g$, we see 
that $(X,Y,g) \in \ZC$. So we have proved the following fact:

\medskip

\noindent{\bf Fact 1.} {\it $p(\ZC)=\XC_\th(n\d_l)^{\mub_m}$.}

\medskip

The next fact follows from the freeness of the action of $\Pb\Gb\Lb(n\d_l)$ on 
$\YC_\th(n\d_l)$ (see Theorem~\ref{theo:smooth}). 

\medskip

\noindent{\bf Fact 2.} {\it If $(X,Y,g)$ and $(X,Y,g')$ belong to $\ZC$, then $g=g'$ and $g^k=g_0$.} 

\medskip

Indeed, the hypothesis implies that $\lexp{g}{(X,Y)}=\lexp{g'}{(X,Y)}$ and $\lexp{g^k}{(X,Y)}=\lexp{g_0}{(X,Y)}$. 
So there exists two non-zero complex numbers $\a$ and $\b$ such that $g'=\a g$ and $g^k=\b g_0$. 
But the condition $g p_{X,Y}= g' p_{X,Y} = g_0 p_{X,Y}$ forces $\a=\b=1$ because $p_{X,Y} \neq 0$. 

\medskip

Now, let $\CC$ denote the set of elements $g \in \Gb\Lb(n\d_l)$ such that $g^k=g_0$. 
Since $g_0$ is central in $\Gb\Lb(n\d_l)$, $\CC$ is a disjoint union of finitely 
many semisimple conjugacy classes (which are the irreducible components of $\CC$). 
We fix $g \in \CC$ such that $q^{-1}(g)\ne \vide$ and we denote by $\CC_g$ its conjugacy class in $\Gb\Lb(n\d_l)$. 
We will describe $q^{-1}(g)$. For this, let $E=\bigoplus_{i \in \ZM/l\ZM} \CM^n$ and 
let $E_j=E^{\xi^{-j}g}$ (for $j \in \ZM/m\ZM$). In other words, $E_j$ is the $\xi^j$-eigenspace of $g$ 
in $E$: it is contained in the $i$-th component of $E$, where $i$ is the element of $\ZM/l\ZM$ such 
that $j \equiv i \mod l$. Let $d_j = \dim_\CM(E_j)$ and let $d=(d_j)_{j \in \ZM/m\ZM}$. 

If $(X,Y,g) \in \ZC$, then $X$ sends $E_j$ to $E_{j-1}$ (we denote by $X_j : E_j \to E_{j-1}$ the induced map) 
while $Y$ sends $E_j$ to $E_{j+1}$ (we denote by $Y_j : E_j \to E_{j+1}$ the induced map). By choosing 
a basis of every $E_j$, the family $(X_j,Y_j)_{j \in \ZM/m\ZM}$ defines 
a representation of the quiver $\Qov_m$, of dimension vector $d$. This means 
that we have defined a map $q^{-1}(g) \longto \Rep(\Qov_m,d)$, which is clearly 
a closed immersion, and whose image is $\YC_{\th[k]}(d)$ (because $\im(p_{X,Y})$ 
is contained in $E_0$). 

Note that by construction we have $d\in \ZM^{\ZM/m\ZM}[n\d_l]$. Moreover, by Lemma \ref{lem:descrX-if-simple}(b) 
and Theorem \ref{theo:smooth} we also have $d\in (\ZM^{\ZM/m\ZM})_+$. This means that $d$ is an element of 
$\EC(k,l,n)=\ZM^{\ZM/m\ZM}[n\delta_l]\cap (\ZM^{\ZM/m\ZM})_+$. Moreover, it is clear that the set of elements $d$ 
that can appear as above (from some $g\in \CC$ such that $q^{-1}(g)\ne \vide$) is $\EC(k,l,n)_{\ne\vide}$.

This shows the following fact.

\medskip

\noindent{\bf Fact 3.} {\it With the above notation, $q^{-1}(g) \simeq \YC_{\th[k]}(d)$ . In particular, 
$q^{-1}(g)$ is irreducible.}

\medskip


Fact 3 and its proof imply the following.

\medskip
\noindent{\bf Fact 4.} {\it We have $p(q^{-1}(\CC_g))=\XC_d$.}

\medskip

Denote by $\pi$ the projection $\pi\colon \YC_\th(n\d_l)\to \XC_\th(n\d_l)$. Fact 2 implies that the map 
$p'\colon\ZC\to\pi^{-1}(\XC_\th(n\d_l)^{\mub_m})$ is bijective. Moreover, since $\pi^{-1}(\XC_\th(n\d_l)^{\mub_m})$ 
is smooth (the group $\Pb\Gb\Lb(n\d_l)$ acts freely on $\YC_\th(n\d_l)$), this bijection is an isomorphism of 
algebraic varieties. Let 
$$
(p')^{-1}\colon \pi^{-1}(\XC_\th(n\d_l)^{\mub_m})\to \ZC
$$ 
be its inverse. 

Consider the morphism $\widetilde\tau\colon \pi^{-1}(\XC_\th(n\d_l)^{\mub_m})\to \Gb\Lb(n\d_l)$ given 
by $\widetilde\tau=q\circ(p')^{-1}$. Since we have $\pi^{-1}(\XC_\th(n\d_l)^{\mub_m})/\!\!/\Gb\Lb(n\d_l)=\XC_\th(n\d_l)^{\mub_m}$, 
we get a morphism $\tau\colon \XC_\th(n\d_l)^{\mub_m}\to \Gb\Lb(n\d_l)/\!\!/\Gb\Lb(n\d_l)$, where the group $\Gb\Lb(n\d_l)$ acts on 
itself by conjugation.  

Moreover, since we clearly have $\tau^{-1}(\CC_g)=p(q^{-1}(\CC_g))$, Fact 4 implies the following.

\medskip
\noindent{\bf Fact 5.} {\it We have $\tau^{-1}(\CC_g)=\XC_d$. In particular $\XC_d$ is closed. }

\medskip

This completes the proof of~(b). Now we prove~(c). 
The group $\Gb\Lb(n\d_l)$ acts transitively on $\CC_g$. So we get an isomorphism 
of varieties 
$$q^{-1}(g)/\!\!/C_{\Gb\Lb(n\d_l)}(g) \longiso q^{-1}(\CC_g)/\!\!/\Gb\Lb(n\d_l).$$
Now, if $d=(\dim_\CM(E^{\xi^{-j}g}))_{j \in \ZM/m\ZM}$, $q^{-1}(g)$ identifies 
with $\YC_{\th[k]}(d)$ and $C_{\Gb\Lb(n\d_l)}(g)$ identifies with $\Gb\Lb(d)$. 
Glueing this with Facts 3 and 4, we get an isomorphism of varieties
$$\XC_{\th[k]}(d) \longiso \XC_d,$$
which is the map $i_k^{(d)}$. This proves~(c).
\end{proof}


\bigskip
\begin{coro}
\label{coro:big-m-C*}
If $k>n$, then we have $\XC_\th(n\d_l)^{\mub_m}=\XC_\th(n\d_l)^{\CM^\times}$.
\end{coro}

\bigskip
\begin{proof}
Take $d \in \EC(k,l,n)_{\ne\vide}$. Since $k>n$, some component of $d$ must be zero. 
This implies that $\CM^\times$ acts trivially on $\XC_{\th[k]}(d)$  because the action of $\CM^\times$ is induced by elements of $\Gb\Lb(d)$.

Since $i_k^{(d)}$ is $\CM^\times$-equivariant and 
$$\XC_\th(n\d_l)^{\mub_m}=\coprod_{d\in \EC(k,l,n)_{\ne\vide}}i_k^{(d)}(\XC_{\th[k]}(d))$$ 
(by Theorem~\ref{theo:fixed-quiver}(b)), the $\CM^\times$-action on the variety $\XC_\th(n\d_l)^{\mub_m}$ is also trivial. 
\end{proof}
\bigskip

\section{Proof of Theorem~\ref{theo:main}}\label{sec:reductions}

\medskip

First, the explicit description of $\ZCB_{\! c}$ whenever $n=1$ (see Example~\ref{ex:cyclique}) 
allows to prove easily Theorem~\ref{theo:main} in this case. 
Also, let us decompose $V$ as
$$V=V^W \oplus V_1 \oplus \cdots \oplus V_r,$$
where the $V_i$ are the non-trivial irreducible components of $V$ as a $\CM W$-module. 
Then $W$ decomposes as $W = W_1 \times \cdots \times W_r$, where $(V_i,W_i)$ is then 
a reflection subquotient (subgroup) of $W$. Let $c_i$ denote the restriction of $c$ to $W_i$. 
Then 
$$\ZCB_{\! c}(V,W) \simeq V^W \times V^{* W} \times \ZCB_{\! c_1}(V_1,W_1) \times \cdots \times \ZCB_{\! c_r}(V_r,W_r)$$
so that the proof of Theorem~\ref{theo:main} is easily reduced to the case where $W$ acts 
irreducibly on $V/V^W$. 

In this case, the smoothness of $\ZCB_{\! c}(V,W)$ implies that $W$ is of type $G_4$ or $G(l,1,n)$ for some 
$l \ge 1$. The case of type $G_4$ will be handled by computer calculations: 
see Section~\ref{sec:g4} for details. This means that we can work under the following hypothesis:

\medskip

\boitegrise{{\bf Hypothesis.} {\it From now on, and until the end of this section, 
we assume that $n \ge 2$, that $V=\CM^n$ and that $W=G(l,1,n)$. We also assume that $\s$ is an $m$-th root 
of unity, where $m=kl$ for some $k \ge 1$.}}{0.75\textwidth}

\medskip

Recall that $G(l,1,n)$ is the group of monomial 
matrices with coefficients in $\mub_l$ (the group of $l$-th root of unity in $\CM^\times$) 
and recall that $n \ge 2$. 

\bigskip

\subsection{Quiver varieties vs Calogero-Moser spaces}\label{sub:quiver-cm}
We fix a primitive $l$-th root of unity $\zeta$. We denote by $s$ the permutation matrix corresponding to the 
transposition $(1,2)$ and we set
$$t=\diag(\zeta,1,\dots,1) \in W.$$
Then $s$, $t$, $t^2$,\dots, $t^{l-1}$ is a set of representatives of conjugacy 
classes of reflections of $W$. We set for simplification
$$a=c_s\qquad\text{and}\qquad k_j=\frac{1}{l} \sum_{i=1}^{l-1} \zeta^{-i(j-1)} c_{t^i}$$
for $j \in \ZM/l\ZM$. Then
\equat\label{eq:k}
k_0+\cdots + k_{l-1} = 0\qquad \text{and}\qquad c_{t^i}=\sum_{j \in \ZM/l\ZM} \zeta^{i(j-1)} k_j
\endequat
for $1 \le i \le l-1$. Finally, if $i \in \ZM/l\ZM$, we set

\equat
\th_i=
\begin{cases}
k_{-i}-k_{1-i} & \text{if $i \neq 0$,}\\
-a +k_0-k_1 & \text{if $i=0$.}\\
\end{cases}
\endequat\label{eq:theta-k}
and $\th=(\th_i)_{i \in \ZM/l\ZM}$. 

The following result is proved in~\cite[Theorem~3.10]{gordon quiver}. (Note that our $k_i$ is related with 
Gordon's $H_i$ via $H_i=k_{-i}-k_{1-i}$.)

\bigskip

\begin{theo}\label{theo:quiver-cm}
With the above notation, there is a $\CM^\times$-equivariant 
isomorphism of varieties
$$\ZCB_{\! c} \longiso \XC_\th(n\d_l).$$
\end{theo}

\bigskip

The theorem above is also true for $n=1$. But in this case, the variety $\ZCB_{\! c}$ has no parameter $a$. On the other hand, the variety $\XC_\th(\d_l)$ is independent of $\theta_0$.

Putting Theorems~\ref{theo:smooth} and~\ref{theo:quiver-cm} together, one gets:

\bigskip

\begin{coro}\label{coro:smooth:gl1n}
The variety $\ZCB_{\! c}$ is smooth if and only if 
$$ a \prod_{\substack{0 \le i \neq j \le l-1 \\ 0 \le r \le n-1}} (k_i-k_j - r a) \neq 0.$$
\end{coro}

\bigskip

\begin{proof}[Proof of Theorem~\ref{theo:main}]
Assume now that $\ZCB_{\! c}$ is smooth and let $\XC$ be an irreducible component of 
$\ZCB_{\! c}^{\mub_m}$. Using the isomorphism of Theorem~\ref{theo:quiver-cm}, we 
see that 
$$\ZCB_{\! c}^{\mub_m} \longiso \XC_\th(n\d_l)^{\mub_m}.$$
So, by Theorem~\ref{theo:fixed-quiver}, there exists $d \in \EC(k,l,n)_{\ne\vide}$ 
such that $\XC \simeq \XC_{\th[k]}(d)$. By Lemma~\ref{lem:orbites}, there exists 
$r \ge 0$ and $w \in W_{m}^\aff$ such that $w(d) = r\d_{m}$. So it follows from 
the isomorphism~(\ref{eq:iso-lusztig}) that
$$\XC \simeq \XC_{w(\th[k])}(r\d_{m})$$
(see also Lemma \ref{lem:descrX-if-simple} and its proof).
Using Theorem~\ref{theo:quiver-cm} in the other way, we see that there exists 
a complex valued function $c'$ on the set of conjugacy classes of reflections 
of $W'=G(m,1,r)$ such that $\XC_{w(\th[k])}(r\d_{m}) \simeq \ZCB_{\! c'}(W')$. 
Moreover, as the action of $W_m^\aff$ on $\CM^{\ZM/m\ZM}$ is linear, 
this implies that the map $c \mapsto c'$ is linear.

This proves almost every statement of Theorem~\ref{theo:main}, except that 
$W'$ can be realized as a reflection subquotient of $(V,W)$. For this, we need the following fact.

\medskip
\noindent{\bf Fact.} {\it We have $kr \le n$.} 

\medskip

Indeed, the element $d-r\d_m$ is in the $W^\aff_m$-orbit of $0$. Then Proposition \ref{prop:core=orbit0} 
implies that $d-r\d_m$ is a residue of some $m$-core. In particular, each coordinates of $d-r\d_m$ must be positive. 
Since the sum of the coordinates of $d$ is $nl$, this gives $nl-rkl\geqslant 0$. This implies $kr\leqslant n$.

Now, using the fact proved above, let $x$ denotes the matrix 
$$x=\diag(\underbrace{\z,1,\dots,1}_{\text{$k$ terms}})\cdot M_{(1,2,\dots,k)} \in \Gb\Lb_k(\CM),$$
where $M_{(1,2,\dots,k)}$ is the permutation matrix associated with the cycle $(1,2,\dots,k)$. Now, let $g$ 
be the matrix
$$g=\diag(\underbrace{x,\dots,x}_{\text{$r$ terms}},\Id_{\CM^{n-rk}}) \in G(l,1,n).$$
Let $V'$ denote the $\xi$-eigenspace of $g$, where $\xi$ is a primitive $m$-th root 
of unity such that $\xi^k=\z$. Then $C_W(g)$ acts on $V'$, $V'$ is of dimension $r$ and, 
if we denote by $K$ the kernel of this action, then $C_W(g)/K \simeq G(m,1,r)$. 
So $(V',G(m,1,r))$ is a reflection subquotient of $(V,W)$.
\end{proof}

\bigskip

\begin{rema}\label{rem:perm-k}
The formula \eqref{eq:theta-k} above yields a bijection between $\CM^{\ZM/l\ZM}$ and the set of $(l+1)$-tuples of 
complex numbers $(a,k_0,\ldots,k_{l-1})$ such that $k_0+k_1+\ldots+k_{l-1}=0$. 
Moreover, the $W^\aff_l$-action on $\CM^{\ZM/l\ZM}$ (see Section \ref{sub:action-weyl}) translates 
to the action on the $(l+1)$-tuples in the following way
$$
s_r(a,k_0,\ldots,k_{l-r},k_{l-r+1},\ldots,k_{l-1})=(a,k_0,\ldots,k_{l-r+1},k_{l-r},\ldots,k_{l-1}),\quad \mbox{ for }r=1,2,\ldots,l-1
$$
and
$$
s_0(a,k_0,k_1,k_2,\ldots,k_{l-1})=(a,k_1+a,k_0-a,k_2,\ldots,k_{l-1}).
$$
In particular, we see that the subgroup $W_l$ of $W^\aff_l$ acts by permutation of the 
parameters $k_0$, $k_1$,\dots, $k_{l-1}$.\finl
\end{rema}

\bigskip

The next result follows from~\cite{BST}, but we provide here a different proof (which 
works only in type $G(d,1,n)$).

\bigskip
\begin{coro}
A permutation of parameters $k_0,k_1,\ldots,k_{l-1}$ does not change the Calogero-Moser space $\ZCB_{\! c}$ 
(up to an isomorphism of algebraic varieties).
\end{coro}

\bigskip
\begin{proof}
The subgroup $W_l$ of $W^\aff_{l}$ stabilizes $\delta_l$. Then by \eqref{eq:iso-lusztig}, for each 
$w\in W_l$, $n\in \ZM_{>0}$ and $\theta\in \CM^{\ZM/l\ZM}$, we have $\XC_\th(n\d_l)\simeq \XC_{w(\th)}(n\d_l)$. 
This proves the statement by Theorem \ref{theo:quiver-cm} and Remark \ref{rem:perm-k}
\end{proof}

\bigskip

\section{Combinatorics}

\medskip

In this subsection, we aim to make the 
statement of Theorem~\ref{theo:main} more precise in the case where $W=G(l,1,n)$: 
we wish to describe precisely the map $c \mapsto c'$ as well as the 
map $i_\XC^* : \Irr(W') \injto \Irr(W)$ in terms of the combinatorics of 
partitions, cores, $l$-quotients, etc.


\bigskip

\subsection{Partitions and cores}
Let $l$ and $n$ be positive integers. 
A \emph{partition} is a tuple $\lambda=(\lambda_1,\lambda_2,\dots,\lambda_r)$ of positive integers 
(with no fixed length) such that $\lambda_1\geqslant\lambda_2\geqslant\dots\geqslant\lambda_r$. 
Set $|\lambda|=\sum_{i=1}^r \lambda_i$. If $|\lambda|=n$, we say that $\lambda$ is a partition of $n$.

Denote by $\PC$ (resp. $\PC[n]$) be the set of all partitions (resp. the set of all partitions of $n$). 
By convention, $\PC[0]$ contains one (empty) partition. We will identify partitions with Young diagrams. 
The partition $\lambda$ corresponds to a Young diagram with $r$ lines such that the $i$th line contains $\lambda_i$ boxes. 
For example the partition $(4,2,1)$ corresponds to the Young diagram

$$
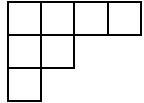
\yng(4,2,1)
$$

Let $b$ be a box of a Young diagram in the line $r$ and column $s$. The $l$-\emph{residue} of the box $b$ is the 
number $s-r$ modulo $l$. (We also say that the integer $s-r$ is the $\infty$-residue of the box $b$). Then we obtain a map 
$$
\Res_l\colon \PC\to \ZM^{\ZM/l\ZM},\qquad \lambda\mapsto \Res_l(\lambda),
$$ 
such that for each $i\in \ZM/l\ZM$ the number of boxes with residue $i$ in $\lambda$ is $(\Res_l(\lambda))_i$. 
(Similarly, we obtain a map $\Res_\infty\colon \PC\to \ZM^{\ZM}$.)

\bigskip

\begin{exemple}
For the partition $\lambda=(4,2,1)$ and $l=3$ the residues of the boxes are
$$
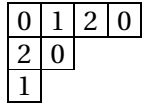
\young(0120,20,1)
$$
In this case we have $\Res_l(\lambda)=(3,2,2)$ because there are three boxes with residue $0$, two boxes with residue 
$1$ and two boxes with residue $2$.\finl
\end{exemple}

\bigskip

We say that a box of a Young diagram is \emph{removable} if it has no box on the right and on the bottom. 
For $\lambda,\mu\in\PC$, we write $\mu\leqslant \lambda$ if the Young diagram of $\mu$ can be obtained from 
the Young diagram of $\lambda$ by removing a sequence of removable boxes.

\bigskip

\begin{defi}
We say that the partition $\lambda$ is an $l$-core if there is no partition $\mu\leqslant\lambda$ such that 
the Young diagram of $\mu$ differs from the Young diagram of $\lambda$ by $l$ boxes with $l$ different $l$-residues.
\end{defi}

\bigskip

See \cite{BJV} for more details about the combinatorics of $l$-cores. Let $\CC_l\subset \PC$ be the set of $l$-cores. 
Set $\CC_l[n]=\PC[n]\cap \CC_l$.

If a partition $\lambda$ is not an $l$-core, then we can get a smaller Young diagram from its Young diagram by 
removing $l$ boxes with different $l$-residues. We can repeat this operation again and again until we get an $l$-core. 
It is well-known that the $l$-core that we get is independent of the choice of the boxes that we remove. Then we get an application
$$
\Core_l\colon \PC\to \CC_l.
$$
If $\mu=\Core_l(\lambda)$, we will say that the partition $\mu$ is {\it the $l$-core} of the partition $\lambda$.

\bigskip

\begin{exemple}
The partition $(4,2,1)$ from the previous example is not a $3$-core because it is possible to remove three bottom boxes. We get 
$$
\young(0120)
$$
But this is still not a $3$-core because we can remove three more boxes and we get 
$$
\young(0)
$$
This shows that the partition $(1)$ is the $3$-core of the partition $(4,2,1)$.\finl
\end{exemple}

\bigskip

As mentioned in Lemma \ref{lem:orbites}, for each $d\in \ZM^{\ZM/l\ZM}$, the $W^\aff_l$-orbit of $d$ 
contains an element of the form $n\delta_l$ for a unique $n\in \ZM$.

\bigskip
\begin{prop}
\label{prop:core=orbit0}
Let $d$ be an element of $\ZM^{\ZM/l\ZM}$. The following statements are equivalent.
\begin{itemize}
\item[(a)] $d$ is of the form $d=\Res_l(\lambda)$ for some $l$-core $\lambda$,
\item[(b)] $d$ is in the $W^\aff_l$-orbit of $0$.
\end{itemize}
\end{prop}
\bigskip
\begin{proof}
Consider the $W^\aff_l$-action on $\CC_l$ as in \cite[Sec.~3]{BJV}. By construction, the map 
$\Res_l\colon \CC_l\to \ZM^{\ZM/l\ZM}$ is $W^\aff_l$-invariant. Moreover, the residue of the empty partition is zero. 
The stabilizer of the empty partition in $W^\aff_l$ is $W_l$ and the stabilizer of $0\in\ZM^{\ZM/l\ZM}$ in  $W^\aff_l$ is also $W_l$. 
This shows that the map $\Res_l$ yields a bijection between the set of $l$-cores and the $W^\aff_l$-orbit of $0$ in $\ZM^{\ZM/l\ZM}$.
\end{proof}
\bigskip

Denote by $\pi$ the obvious surjection $\pi\colon\ZM^{\ZM/l\ZM}\to \ZM^{\ZM/l\ZM}/{\ZM\delta_l}$.
\begin{prop}
\label{prop-bij_cores-roots}
The chain of maps
$$\diagram
\CC_l \rrto^{\DS{\Res_l}}&& \ZM^{\ZM/l\ZM} \rrto^{\DS{\pi}} && \ZM^{\ZM/l\ZM}/{\ZM\delta_l}
\enddiagram
$$
yields a bijection between $\CC_l$ and $\ZM^{\ZM/l\ZM}/{\ZM\delta_l}$.
\end{prop}
\bigskip
\begin{proof}
We have seen in the proof of Proposition \ref{prop:core=orbit0} that $\Res_l$ yields a bijection 
between $\CC_l$ and the $W^\aff_l$-orbit of $0\in \ZM^{\ZM/l\ZM}$. This proves the statement because the 
restriction of $\pi$ to the $W^\aff_l$-orbit of $0\in \ZM^{\ZM/l\ZM}$ is bijective by Lemma \ref{lem:orbites}.
\end{proof}

\bigskip

For $d=(d_i)_{i\in \ZM/l\ZM}\in \ZM^{\ZM/l\ZM}$ we set $|d|=\sum_{i\in \ZM/l\ZM}d_i$. 
Set $\ZM^{\ZM/l\ZM}[n]=\{d\in \ZM^{\ZM/l\ZM}; |d|=n \}$. Denote by  $\CC_l[\equiv n]$ the 
subset of $\CC_l$ that contains only $l$-cores with the number of boxes congruent to $n$ modulo $l$.
Proposition \ref{prop-bij_cores-roots} yields the following bijection 
$\varepsilon\colon \ZM^{\ZM/l\ZM}[n] \longbij \CC_l[\equiv n]$.

\bigskip

\begin{lem}
\label{lem:+_equiv_<=n}
Assume $d\in \ZM^{\ZM/l\ZM}[n]$. Then $d$ is in $(\ZM^{\ZM/l\ZM})_+$ if and only if $|\varepsilon(d)|\leqslant n$.
\end{lem}
\bigskip
\begin{proof}
By construction, $\varepsilon(d)$ is the unique $l$-core $\lambda\in \CC_l[\equiv n]$ such that $\Res_l(\lambda)$ is 
congruent to $d$ modulo $\d_l$. Write $d=\Res_l(\lambda)+r\delta_l$. Since by Proposition \ref{prop:core=orbit0}, 
the element $\Res_l(\lambda)$ is in the $W^\aff_l$-orbit of $0$, then $d$ is in the $W^\aff_l$-orbit of $r\delta_l$.  
Then, by definition, the element $d$ is in $(\ZM^{\ZM/l\ZM})_+$ if and only if $r\geqslant 0$. Now, let us count 
the sum of the coordinates in the sides of the equality $d=\Res_l(\lambda)+r\delta_l$. We obtain $n=|\lambda|+rl$. 
This shows that we have $r\geqslant 0$ if and only if $|\lambda|\leqslant n $.
\end{proof}

\bigskip

Let $\nu$ be an $l$-core.
Denote by $\PC_{l,\nu}$ the set of partitions with the $l$-core $\nu$. In particular we denote by $\PC_{l,\vide}$ 
the set of partitions with trivial $l$-core. Let $\PC^l$ be the set of $l$-partitions (i.e., the set of $l$-tuples of partitions). 
If $\l$ is a partition, we denote by $\b_l(\l)$ its {\it $l$-quotient} (see for example \cite[Sec.~2.2]{Leeuwen}). 

\bigskip

\begin{lem}
\label{lem:l-quot}
If $\nu$ is an $l$-core, then $\b_l$ restricts to a 
bijection $\beta_{l,\nu}\colon\PC_{l,\nu} \longbij \PC^l$.
\end{lem}

\bigskip

\subsection{Description of $\EC(k,l,n)$ in terms of cores}
Let $k$ be a positive integer. Set $m=kl$. Set also $(\CC_m)_{l,\nu}=\PC_{l,\nu}\cap \CC_m$.


Recall that we have a bijection $\varepsilon\colon \ZM^{\ZM/m\ZM}[nl]\longbij \CC_m[\equiv nl]$.  Assume $d\in \ZM^{\ZM/m\ZM}[nl]$.

\bigskip

\begin{lem}
\label{lem:ndl_equiv_triv}
We have $d\in \ZM^{\ZM/m\ZM}[n\delta_l]$ if and only if the $l$-core of $\varepsilon(d)$ is trivial, where the 
set $\ZM^{\ZM/m\ZM}[n\delta_l]$ is as in Section \ref{sub:fixed}.
\end{lem}

\bigskip
\begin{proof}
By definition, $\varepsilon(d)$ is the unique $m$-core $\lambda\in \CC_m[\equiv nl]$ such that $\Res_m(\lambda)$ is 
equal to $d$ modulo $\d_m$. On the other hand, by Proposition \ref{prop:core=orbit0}, the partition $\lambda$ is an $l$-core 
if and only if $\Res_l(\lambda)$ is a multiple of $\d_l$. By commutativity of the following diagram,
$$
\diagram    
& \PC  \dlto_{\DS{\Res_m}} \drto^{\DS{\Res_l}} &\\
  \ZM^{\ZM/m\ZM} \rrto & & \ZM^{\ZM/l\ZM} \\
\enddiagram
$$
this condition is clearly equivalent to $d\in \ZM^{\ZM/m\ZM}[n\delta_l]$.
\end{proof}

Recall that the set $\EC(k,l,n)$ was defined as the intersection $\EC(k,l,n)=\ZM^{\ZM/m\ZM}[n\delta_l]\cap (\ZM^{\ZM/m\ZM})_+$. 
Combining Lemmas \ref{lem:+_equiv_<=n} and \ref{lem:ndl_equiv_triv}, we obtain the following.

\bigskip
\begin{lem}
\label{lem:bij_eps'}
The bijection $\varepsilon\colon \ZM^{\ZM/m\ZM}[nl]\simeq \CC_m[\equiv nl]$ restricts to a bijection
$$
\varepsilon'\colon \EC(k,l,n)\to \CC_m[\equiv nl]\cap \CC_m[\leqslant nl]\cap (\CC_m)_{l,\vide}.
$$
\end{lem}

\bigskip

For $\lambda=(\lambda^0,\dots,\lambda^{l-1})\in\PC^l$ we set $|\lambda|= |\lambda^0|+\dots+|\lambda^{l-1}|$. 
Set also $\PC^l[n]=\{\lambda\in\PC^l;~ |\lambda|=n\}$.
We say that an element of $\PC^l[n]$ is an $l$-partition of $n$.

Now recall the bijection $\beta_{l,\nu}\colon\PC_{l,\nu} \longbij \PC^l$ obtained by restriction from $\b_l$ 
as in Lemma \ref{lem:l-quot}. 
By construction, it has the following properties.

\bigskip
\begin{lem}
\label{lem:propetries-lquot}

\begin{itemize}
\itemth{a} For each $\lambda\in \PC_{l,\vide}$, we have $\lambda\in \CC_m$ if and only if $\beta_{l,\vide}(\lambda)\in (\CC_k)^l$.

\itemth{b} For each $\lambda\in \PC_{l,\nu}$, we have $|\lambda|=l\cdot|\beta_{l,\nu}(\lambda)|+|\nu|$.
\end{itemize}



\end{lem}

\bigskip

Part (a) of the lemma above shows that $\beta_{l,\vide}$ can be restricted to a bijection $(\CC_m)_{l,\vide}\simeq (\CC_k)^l$. 
Moreover, part (b) shows that new bijection restricts to a bijection 
$$
\CC_m[\equiv nl]\cap \CC_m[\leqslant nl]\cap (\CC_m)_{l,\vide}\simeq (\CC_k)^l[\equiv n,\leqslant n],
$$
where
$$
(\CC_k)^l[\equiv n,\leqslant n]=\{\lambda\in (\CC_k)^l;~|\lambda |\leqslant n ~\text{and} ~ |\lambda |\equiv n\mod k\}.
$$

Combining this with Lemma \ref{lem:bij_eps'}, we get a bijection
$$
\d \colon \EC(k,l,n)\to (\CC_k)^l[\equiv n,\leqslant n].
$$

\bigskip

\subsection{Parametrization of the $\mub_m$-fixed points}

We have proved in Theorem~\ref{theo:fixed-quiver} that the subset $\EC(k,l,n)_{\ne\vide}$ of $\EC(k,l,n)$ 
parametrizes the irreducible components of $\XC_\th(n\d_l)^{\mub_m}$. (But we will show in 
Corollary~\ref{coro:always-not-empty} that we have $\EC(k,l,n)_{\ne\vide}=\EC(k,l,n)$.) Then the bijection $\d$ 
above gives another parametrization of the irreducible components of $\XC_\th(n\d_l)^{\mub_m}$ in terms of 
$(\CC_k)^l[\equiv n,\leqslant n]$. For each $\lambda\in (\CC_k)^l[\equiv n,\leqslant n]$, we set $\XC_\lambda=\XC_{\d^{-1}(\lambda)}$. 

Recall the map $\Core_k \colon \PC\to \CC_k$. It yields a map $\Core_k \colon \PC^l\to (\CC_k)^l$ (component by component).  
Note that the set $(\CC_k)^l[\equiv n,\leqslant n]$ considered above is nothing else but the image of $\PC^l[n]$ by $\Core_k$. 
In other words, the set $(\CC_k)^l[\equiv n,\leqslant n]$ is the same thing as the set of $k$-cores of $l$-partitions of $n$. 

This new parametrization has the following nice property.

\bigskip
\begin{lem}
\label{lem:param-compatible}
Fix positive integers $k_1$ and $k_2$ such that $k_1$ divides $k_2$. Set $m_1=k_1l$ and $m_2=k_2l$. 
Fix $\lambda_1\in (\CC_{k_1})^l[\equiv n,\leqslant n]$ and $\lambda_2\in(\CC_{k_2})^l[\equiv n,\leqslant n]$. 
Then we have $\XC_{\lambda_2}\subset \XC_{\lambda_1}$
if and only if $\Core_{k_1}(\lambda_2)=\lambda_1$.
\end{lem}
\bigskip
\begin{proof}
Consider the map 
$$
p\colon \ZM^{\ZM/{m_2}\ZM}\to \ZM^{\ZM/{m_1}\ZM}, \qquad p(d)_i=\sum_{j\equiv i \mod m_1}d_j.
$$
Let $d^1\in\EC(k_1,l,n)$ and $d^2\in\EC(k_2,l,n)$ be such that $\d(d^1)=\lambda_1$ and $\d(d^2)=\lambda_2$. 
We have $\XC_{d^1}=\XC_{\lambda_1}$ and $\XC_{d^2}=\XC_{\lambda_2}$. It is clear from the definition that we have 
$\XC_{d^2}\subset \XC_{d^1}$ if and only if $p(d^2)=d^1$. Then the statement follows from the commutativity of 
the following diagram
$$
\diagram
\EC(k_2,l,n) \rrto^{\DS{p}} \dto && \EC(k_1,l,n)\dto \\
\CC_{m_2}[\equiv nl]\cap \CC_{m_2}[\leqslant nl]\cap (\CC_{m_2})_{l,\vide} \dto \rrto^{\DS{\Core_{m_1}}} &&
\CC_{m_1}[\equiv nl]\cap \CC_{m_1}[\leqslant nl]\cap (\CC_{m_1})_{l,\vide} \dto \\
(\CC_{k_2})^l[\equiv n,\leqslant n] \rrto^{\DS{\Core_{k_1}}}  && (\CC_{k_1})^l[\equiv n,\leqslant n],
\enddiagram
$$
where the vertical maps are the bijections discussed above.
\end{proof}

\bigskip

\begin{prop}
\label{prop:param-Cfixed}
The $\CM^\times$-fixed points in $\XC_\th(n\d_l)$ are parametrized by the set $\PC^l[n]$. 
This parametrization $\XC_\th(n\d_l)^{\CM^\times}=\{p_\mu;~\mu\in \PC^l[n]\}$ can 
be chosen in such a way that for each $k\in \ZM$, each $\lambda\in (\CC_k)^l[\equiv n,\leqslant n]$ 
and each $\mu\in \PC^l[n]$ we have $p_\mu\in \XC_\lambda$ if and only if $\lambda$ is the $k$-core of $\mu$.
\end{prop}
\bigskip
\begin{proof}
Assume $k>n$. Then by Corollary \ref{coro:big-m-C*}, we have $\XC_\th(n\d_l)^{\mub_m}=\XC_\th(n\d_l)^{\CM^\times}$. 
Moreover, the set $\XC_\th(n\d_l)^{\CM^\times}$ is finite. We already know that $\XC_\th(n\d_l)^{\mub_m}$ is in bijection 
with a subset of $(\CC_{k})^l[\equiv n,\leqslant n]$. Moreover, $k>n$ also implies that we 
have $(\CC_{k})^l[\equiv n,\leqslant n]=\PC^l[n]$. This shows that $\XC_\th(n\d_l)^{\mub_m}=\XC_\th(n\d_l)^{\CM^\times}$ 
is in bijection with a subset of $\PC^l[n]$. Moreover, by Lemma \ref{lem:param-compatible}, this bijection is independent 
of the choice of $k$.  But it is well-known (see \cite{gordon}) that $\XC_\th(n\d_l)^{\mub_m}$ and $\PC^l[n]$ have the 
same cardinalities. This shows that the bijection above is a bijection between $\XC_\th(n\d_l)^{\mub_m}$ and $\PC^l[n]$ 
(not just a subset of $\PC^l[n]$).

The second statement follows from Lemma \ref{lem:param-compatible}.
\end{proof}

\bigskip
\begin{coro}
\label{coro:always-not-empty}
For each $\lambda\in (\CC_{k})^l[\equiv n,\leqslant n]$, we have $\XC_\lambda\ne \vide$. 
In particular, we have $\EC(k,l,n)=\EC(k,l,n)_{\ne\vide}$.
\end{coro}
\bigskip
\begin{proof}
Assume $\lambda\in (\CC_{k})^l[\equiv n,\leqslant n]$. Then there exists $\mu\in \PC^l[n]$ 
such that $\Core_k(\mu)=\lambda$. Then $\XC_\lambda$ is not empty because we have $p_\mu\in \XC_\lambda$ by Proposition \ref{prop:param-Cfixed}.
\end{proof}

\bigskip

\subsection{$\CM^\times$-fixed points}
In the previous section we have constructed a parametrization of the $\CM^\times$-fixed points 
in $\XC_\th(n\d_l)$ by $\PC^l[n]$. Another parametrization of the $\CM^\times$-fixed points in 
$\XC_\th(n\d_l)$ by the same set is done by Gordon \cite{gordon}. This parametrization is given 
in terms of baby Verma modules.

It is not obvious at all how to compare these two parametrizations. But this can be deduced 
from the main theorem in \cite{Przez}.

\bigskip
For $\lambda=(\lambda^0,\dots,\lambda^{l-1})\in\PC^l$ we set $\lambda^\flat=(\lambda^{l-1},\dots,\lambda ^0)$. 
\begin{prop}\label{prop:compare-Gordon}
The parametrization of the $\CM^\times$-fixed points in $\XC_\th(n\d_l)$ given in 
Proposition~\ref{prop:param-Cfixed} differs from Gordon's parametrization by the twist
$$
\PC^l[n]\to \PC^l[n],\quad \lambda\mapsto\lambda^\flat.
$$ 
\end{prop}
\bigskip
\begin{proof}
We have constructed in Lemma \ref{lem:bij_eps'} a parametrization of the irreducible components of 
$\XC_\th(n\d_l)^{\mu_m}$ by the set $\CC_m[\equiv nl]\cap \CC_m[\leqslant nl]\cap (\CC_m)_{l,\vide}$. 
Similarly to the proof of Proposition~\ref{prop:param-Cfixed}, this yields a bijection between  
$\XC_\th(n\d_l)^{\CM^\times}$ and $\PC[nl]\cap \PC_{l,\vide}$.

Indeed, if $k>n$, then we have $\XC_\th(n\d_l)^{\CM^\times}=\XC_\th(n\d_l)^{\mu_m}$ by Corollary~\ref{coro:big-m-C*} and 
$$
\CC_m[\equiv nl]\cap \CC_m[\leqslant nl]\cap (\CC_m)_{l,\vide}=\PC[nl]\cap \PC_{l,\vide}.
$$ 
Moreover, an argument similar to the proof of Proposition \ref{prop:param-Cfixed} shows that the obtained bijection 
$\XC_\th(n\d_l)^{\CM^\times}\simeq\PC[nl]\cap \PC_{l,\vide}$ is independent of $k$.

On the other hand, \cite[Theorem~1.1~a)]{Przez} also constructs a bijection between $\XC_\th(n\d_l)^{\CM^\times}$ and 
$\PC[nl]\cap \PC_{l,\vide}$. It is clear from the construction that this bijection coincides with ours.

Theorem \cite[Theorem~1.2]{Przez} claims that the bijection
$$
\PC[nl]\cap \PC_{l,\vide}\to \PC^l[n],\quad \lambda\mapsto(\beta_{l,\vide}(\lambda))^\flat
$$
identifies the parametrization of the $\CM^\times$-fixed points by the set $\PC[nl]\cap \PC_{l,\vide}$ with 
Gordon's parametrization of the $\CM^\times$-fixed points by the set $\PC^l[n]$.

On the other hand, the parametrization of the $\CM^\times$-fixed points by $\PC[nl]\cap \PC_{l,\vide}$ 
composed with the $l$-quotient bijection $\beta_{l,\vide}$ is nothing else but the parametrization of the $\CM^\times$-fixed 
points by $\PC^l[n]$ from Proposition \ref{prop:param-Cfixed}.
\end{proof}

\bigskip
\begin{rema}
\label{rem:od-CM}
The parametrization of the points of $\XC_\th(n\d_l)^{\CM^\times}$ given in Proposition \ref{prop:param-Cfixed} 
is constructed as a parametrization of the irreducible components of $\XC_\th(n\d_l)^{\mu_m}$ for a very big $m$. 
This parametrization can be seen in an "$m$-independent way" if we replace the quiver $\overline{Q}_m$ with a very 
big $m$ by an infinite quiver. This can be done in the following way.

 Let $\overline Q_\infty$ be the quiver defined in the same way as $\overline Q_l$ with respect to the vertex 
 set $\ZM$ instead of $\ZM/l\ZM$. The dimension vectors of the quiver $\overline Q_\infty$ are in $\ZM^\ZM$. 
Set
$$
\ZM^\ZM_{\rm fin}=\{\hd\in \ZM^\ZM;~\mbox{$\hd$ has finitely many non-zero components}\}.
$$ 
Consider the map
$$
p\colon \ZM^\ZM_{\rm fin}\to \ZM^{\ZM/l\ZM},\quad (p(\hd))_i=\sum_{j\equiv i\mod l}\hd_j.
$$

For each $\hd\in \ZM^\ZM_{\rm fin}$, we have a linear map 
$i^{(\hd)}_{\infty}\colon \Rep(\overline Q_\infty, \hd)\to \Rep(\overline Q_l,p(\hd))$, 
defined in the same way is in Section \ref{sub:fixed}. Now, for each $\theta\in \CM^{\ZM/l\ZM}$ 
we consider the element $\theta[\infty]\in \CM^\ZM$ given by $\theta[\infty]_i=\theta_{i\mod l}$. 
Then we obtain a $\CM^\times$-invariant morphism of algebraic varieties 
$i^{(\hd)}_\infty\colon \XC_{\theta[\infty]}(\hd)\to \XC_\theta(p(\hd))$. 
But the variety $\XC_{\theta[\infty]}(\hd)$ is obviously $\CM^\times$-stable because the $\CM^\times$ 
action is induced by elements of $\GL(\hd)$. Since $\XC_{\theta[\infty]}(\hd)$ is connected, 
it is a singleton (if it is not empty). Set 
$$
\EC(\infty, l,d)=\{\hd\in \ZM^\ZM_{\rm fin};~p(\hd)=d,\XC_{\theta[\infty]}(\hd)\ne \vide\}.
$$
If $\hd\in \EC(\infty, l,d)$, the image of $i^{(\hd)}_\infty$ is a $\CM^\times$-fixed point in 
$\XC_\theta(d)$. Let us call this point $p_{\hd}$.

Now, we assume $d=n\d_l$, $n\geqslant 0$ and that $\XC_\theta(d)$ is smooth. The construction of 
the parametrization of the $\CM^\times$-fixed points in $\XC_\th(n\d_l)$ given in 
Proposition~\ref{prop:param-Cfixed} implies that each $\CM^\times$-fixed point $p_\mu$, $\mu\in \PC^l[n]$, 
in $\XC_\th(n\d_l)$ is of the form $p_{\hd}$ for a unique $\hd\in \EC(\infty, l,n\d_l)$. Moreover, this $\hd$ 
is given by $\hd=\Res_\infty(\beta_{l,\vide}^{-1}(\mu))$.\finl
\end{rema}

\bigskip
\begin{lem}
Let $d\in \ZM^{\ZM/l\ZM}$ be such that 
the pair $(d,\theta)$ is smooth (see Definition \ref{def:smooth-pair}). 
Then by \eqref{eq:iso-lusztig}, the variety $\XC_\theta(d)$ is smooth and non-empty.
Let $\nu$ be the $l$-core such that we have $d=n\d_l+\Res_l(\nu)$ (see Proposition \ref{prop:core=orbit0}). 
Then we have the following. 

\begin{itemize}
\item[(a)] The  $\CM^\times$-fixed points in $\XC_{\theta}(d)$ are exactly $p_{\hd}$, $\hd\in \EC(\infty, l,d)$ (without repetition).

\item[(b)] We have $\EC(\infty, l,d)=\{\Res_\infty(\lambda);~\lambda\in \PC_{l,\nu}[|d|]\}$.
\end{itemize}
\end{lem}
\bigskip
\begin{proof}
The case $d=n\d_l$ follows from Remark \ref{rem:od-CM}. But the proof in the general case is completely the same. 
It is just important to know that under the given assumption on $(d,\theta)$, each representation in $\XC_\theta(d)$ is simple.
\end{proof}

\bigskip

The lemma above yields parametrizations of $\XC_\theta(d)^{\CM^\times}$ by $\EC(\infty, l,d)$ and by $\PC_{l,\nu}[|d|]$. 
Then the $l$-quotient bijection $\beta_{l,\nu}\colon \PC_{l,\nu}[|d|]\to \PC^l[n]$ 
(see Lemma~\ref{lem:propetries-lquot}(b)) yields also a parametrization of $\XC_\theta(d)^{\CM^\times}$ by $\PC^l[n]$. 
For $\mu\in \PC^l[n]$ we denote by $p_\mu$ the corresponding  $\CM^\times$-fixed point of $\XC_\theta(d)$ (i.e., we 
have $p_\mu=p_{\hd}$ if $\hd=\Res_\infty(\beta_{l,\nu}^{-1}(\mu))$. We can write $p_\mu^d$ instead of $p_\mu$ to stress 
that it is a point in $\XC_\theta(d)^{\CM^\times}$.

The group $W^\aff_l$ acts on $\PC^l[n]$ by permutation of components. More precisely, for each $i\in \ZM/l\ZM$, the element $s_i$ exchanges the components $\mu^i$ and $\mu^{i+1}$ of $\mu=(\mu^0,\dots,\mu^{l-1})$. It is clear that the kernel of the canonical surjection $W^\aff_l\to W_l$ acts trivially. 

\bigskip
\begin{lem}
\label{lem:fixes-same-lpart}
Let $(d,\theta)$ be a smooth pair (see Definition \ref{def:smooth-pair}) 
in $\ZM^{\ZM/l\ZM}\times \CM^{\ZM/l\ZM}$. Assume that $\theta_i\ne 0$. Then for every $\mu\in \PC^l[n]$, the bijection 
$\XC_{\theta}(d)\simeq \XC_{s_i(\theta)}(s_i(d))$ given by \eqref{eq:iso-lusztig} sends the $\CM^\times$-point 
$p_\mu^{d}\in\XC_{\theta}(d) $ to $p_{s_i(\mu)}^{s_i(d)}\in \XC_{s_i(\theta)}(s_i(d))$. 
\end{lem}
\bigskip
\begin{proof}
There is an action of $W^\aff_l$ on $\ZM^\ZM_{\rm fin}$ given by
$$
(s_i(\hd))_j=
\left\{\begin{array}{lll}
\d_{j0}+\hd_{j+1}+\hd_{j-1}-\hd_j &\mbox{ if }i\equiv j \mod l,\\
\hd_j &\mbox{ if }i\not\equiv j \mod l.
\end{array}\right.
$$
It is clear that this action lifts the $W^\aff_l$-actions on $\ZM^{\ZM/l\ZM}$ with respect to the map 
$p\colon \ZM^\ZM_{\rm fin}\to \ZM^{\ZM/l\ZM}$.

Let $\nu$ and $\nu'$ be the $l$-cores such that we have $d=\Res_l(\nu)+n\d_l$ and 
$s_i(d)=\Res_l(\nu')+n\d_l$.  For each $\hd\in \EC(\infty, l,d)$, the following diagram is commutative
$$
\diagram
\XC_{\theta[\infty]}(\hd) \rrto \ddto_{\DS{i^{(\hd)}_\infty}} && \XC_{s_i(\theta)[\infty]}(s_i(\hd))\ddto^{\DS{{i^{(s_i(\hd))}_\infty}}}\\
&&\\
\XC_\theta(d)                             \rrto&& \XC_{s_i(\theta)}(s_i(d)). \\
\enddiagram
$$
where the horizontal maps are the isomorphisms given by \eqref{eq:iso-lusztig}. This implies that the 
isomorphism~\eqref{eq:iso-lusztig} identifies the $\CM^\times$-fixed point $p_{\hd}$ with the $\CM^\times$-fixed 
point $p_{s_i(\hd)}$. Moreover, it follows from \cite[Proposition~4.1.3]{Leeuwen} 
that the following diagram is commutative
$$
\diagram
\PC^l \rrto^{\DS{\beta_{l,\nu}^{-1}}} \ddto_{\DS{s_i}} && \PC_{l,\nu} \rrto^{\DS{\Res_\infty}} && \ZM^\ZM_{\rm fin}\ddto^{\DS{s_i}} \\
&&&&\\
\PC^l \rrto^{\DS{\beta_{l,\nu'}^{-1}}} && \PC_{l,\nu'} \rrto^{\DS{\Res_\infty}} && \ZM^\ZM_{\rm fin}\\
\enddiagram
$$
This implies that \eqref{eq:iso-lusztig} identifies the $\CM^\times$-fixed point $p^d_\mu$ with the $\CM^\times$-fixes point $p_{s_i(\mu)}^{s_i(d)}$. 
\end{proof}
\bigskip

\begin{rema}
\label{rem:W'-W}
Fix $d\in \EC(k,l,n)$. Let $r$ be such that $d$ is in the $W^\aff_m$-orbit of $r\d_m$. Let $\nu$ be the $m$-core such that 
$d=r\d_m+\Res_m(\nu)$. Consider the irreducible component $\XC_d$ of $\XC(n\d_l)^{\mu_m}$. By Theorem \ref{theo:fixed-quiver} 
and the proof of Theorem \ref{theo:main}, we have an isomorphism
$$
\XC_d\simeq \XC_{\theta'}(r\d_m)
$$
for some $\theta'$.
This isomorphism depends on a choice of an element $w\in W^\aff_m$ such that $w(d)=r\d_m$ (see the proof of Theorem \ref{theo:main}) and of a choice of a reduced expression of $w$ that allows to apply a sequence of isomorphisms \eqref{eq:iso-lusztig}. We will assume that the element $w$ is chosen as in the proof of Lemma \ref{lem:orbites}. In particuler $w$ is in the kernel of the canonical surjection $W^\aff_m\to W_m$. This choice implies that the bijection $\XC_{\theta[k]}(d)\simeq \XC_{w(\theta[k])}(r\d_m)$ yields the following bijection on the $\CM^\times$-fixed points (see Lemma \ref{lem:fixes-same-lpart}): 

$$
\XC^{\CM^\times}_{\theta[k]}(d)\to \XC^{\CM^\times}_{w(\theta[k])}(r\d_m),\qquad p^{d}_\mu \mapsto p^{r\d_m}_\mu.
$$

The inclusion $\XC_{\theta'}(r\d_m)^{\CM^\times}\subset \XC_{\theta}(n\d_l)^{\CM^\times}$ yields a map
\begin{equation}
\label{eq:W'-W}
\PC^m[r]\to\PC^l[n].
\end{equation}
Let us describe this map.

We have the following bijections obtained from the $l$-quotient (or $m$-quotient) map (see also 
Lemma~\ref{lem:propetries-lquot}(b))
$$
\PC^m[r]\simeq \PC_{m,\nu}[mr+|\nu|],\qquad \PC^l[n]\simeq \PC_{l,\vide}[nl].
$$ 
Note that we have $mr+|\nu|=|d|=nl$.  Moreover, we have $\PC_{m,\nu}[mr+|\nu|]\subset \PC_{l,\vide}[nl]$ because 
every partition with m-core $\nu$ has a trivial $l$-core (since the $l$-core of $\nu$ is trivial). Then under 
the isomorphisms above, the inclusion $\PC_{m,\nu}[mr+|\nu|]\subset \PC_{l,\vide}[nl]$ is exactly the map \eqref{eq:W'-W}. 
Indeed, this follows from Lemma \ref{lem:fixes-same-lpart} (and the choice of $w$ above) and from the commutativity of the following diagram 
$$
\diagram
& \XC_{\theta[\infty]}(\hd)  \dlto_{\DS{i^{(\hd)}_\infty}} \drto^{\DS{i^{(\hd)}_\infty}} &\\
\XC_{\theta[k]}(d) \rrto^{\DS{i^{(d)}_k}} & & \XC_{\theta}(n\d_l) \\
\enddiagram
$$
for 
each $\hd\in \EC(\infty,m,d)$.

Now, let us describe the inclusion map $\PC_{m,\nu}[mr+|\nu|]\to \PC_{l,\vide}[nl]$ as a map $\PC^m[r]\to \PC^l[n]$. Set $\gamma=\beta_{l,\vide}(\nu)$. Then, by Lemma \ref{lem:propetries-lquot}, we have $\gamma=(\gamma^0,\dots,\gamma^{l-1})\in(\CC_k)^l$. For each $\lambda\in\PC^l[n]$ such that $\Core_k(\lambda)=\gamma$, we define $\beta_{k,\gamma}(\l)$ to be the $m$-partition 
$\mu=(\mu^0,\mu^1,\dots,\mu^{m-1})$ defined by 
$$\beta_{k,\gamma^i}(\l^i)=(\mu^i,\mu^{i+l},\dots,\mu^{i+(k-1)l})$$
for all $i \in \{0,1,\dots,l-1\}$. The map $\beta_{k,\gamma}$ is an isomorphism between $\PC^l_{k,\gamma}[n]$ and $\PC^m[r]$, where $\PC^l_{k,\gamma}[n]= \{\l\in\PC^l[n];~\Core_k(\lambda)=\gamma\}$. The following diagram is commutative
$$
\begin{CD}
\PC_{m,\nu}[mr+|\nu|] @>>> \PC_{l,\vide}[nl]\\
@V{\beta_{m,\nu}}VV       @VV\beta_{l,\vide}V\\                    
\PC^m[r]               @>{\beta_{k,\gamma}^{-1}}>> \PC^l[n]\\   
\end{CD}
$$
This implies that the map $\PC^m[r]\to \PC^l[n]$ obtained from the inclusion $\XC_{\theta'}(r\d_m)^{\CM^\times}\subset \XC_{\theta}(n\d_l)^{\CM^\times}$ is given by $\beta^{-1}_{k,\gamma}$.
\finl 
\end{rema}

\bigskip

\subsection{Explicit construction of $c'$}

We know from Theorem \ref{theo:fixed-quiver} that each irreducible component $\XCB$ of $\XC_\th(n\d_l)^{\mub_m}$ 
is of the form $\XCB=\XC_d$. Moreover, by Theorem \ref{theo:main} there exists 
a reflection subquotient $(V',W')$ of $(V,W)$ and a complex-valued map $c'\colon W'\to \CM$ such that 
$\XCB \simeq \ZCB_{\! c'}(V',W')$. The reflection subquotient was constructed in the proof of Theorem~\ref{theo:main}. 
Now we give an explicit construction of $c'$.

Recall from~\S\ref{sub:Cher} that the parameter $c$ is determined by the sequence of parameters $(a,k_0,k_1,\ldots, k_{l-1})$. 
Similarly, the parameter $c'$ is determined by $(a',k'_0,k'_1,\ldots, k'_{m-1})$. We are going to give explicit expressions 
for $(a',k'_0,k'_1,\ldots, k'_{m-1})$ in terms of $d \in \EC(k,l,m)$ and $(a,k_0,k_1,\ldots, k_{l-1})$.

First, we recall the construction of $c'$ in the proof of Theorem \ref{theo:main}. There is an element $w\in W_m^\aff$ 
such that $w(d)$ is a multiple of $\delta_m$. Then we put $\theta'=w(\theta[k])$.  Finally, we can obtain 
$(a',k'_0,\ldots,k'_{m-1})$ from $\theta'$ using \eqref{eq:k} and \eqref{eq:theta-k}.

This construction of $c'$ is not canonical because there is no preferable choice of an element $w$. Note that $\delta_m$ 
is stable by the subgroup $W_m$ of $W^{\aff}_m$.  So the element $w$ can be replaced by any elements of $W_mw$. 
(By Remark \ref{rem:perm-k}, a different choice of $w$ yields just a permutation of the parameters $k'_0,k'_1,\ldots,k'_{m-1}$.)  
The affine Weyl group $W^\aff_m$ is of the form $W_m\ltimes R_m$, where $R_m$ is the root lattice for $\mathfrak{sl}_m$.  
In particular, there is an element of $w'\in W_mw$ that is in $R_m$. More precisely, we have $w'=t_\alpha$, where 
$\alpha=d-d_0\delta_m$ (see the proof of Lemma \ref{lem:orbites}).

Then, by Lemma \ref{lem:affWeyl_translation-theta}, we have 
$w'(\theta[k])=t_\alpha(\theta[k])=\theta[k]+\Sigma(\th[k])\cdot {\overline\alpha}=\theta[k]+k\Sigma(\th)\cdot {\overline d}$. (Where the notation $\overline\alpha$ and $\overline d$ is as in Section \ref{sub:action-weyl}.)
This means that we have
$$
\theta'_i=(\theta[k])_{i}+k\Sigma(\th)(2d_i-d_{i-1}+d_{i+1})=\theta_{(i\mod l)}-ka(2d_i-d_{i-1}+d_{i+1})
$$

Passing back to the $k_i$-notations, we get $a'=ka$ and

$$
k'_i=k_{(i\mod l)}+a\left(\left\lfloor \frac{i-1}{l} \right\rfloor-\frac{k-1}{2}+k(d_{1-i}-d_{-i})\right)
$$
for $i= 1,\ldots,m-1, m$. (Here $k'_0=k'_m.)$

\bigskip
\begin{exemple}\label{ex:parametre}
Let us make the above formula explicit in the case $l=1$. In this case we have $m=k$.
Then we get $\theta=\theta_0=-a$ and $\theta[k]=(-a,-a,\ldots -a)$ ($k$ times).
This gives 
$$
\theta'_i=-a-ak(2d_i-d_{i-1}-d_{i+1}),
$$
for $i\in \ZM/k\ZM$.

$$
k'_i=a\left(i-\frac{k+1}{2}+k(d_{1-i}-d_{-i})\right)
$$
for $i=1,\ldots, k-1, k$. (Here $k_0=k_k.)$\finl
\end{exemple}

\bigskip

\begin{exemple}
Assume $l=2$, $k=2$ and $d=(0,0)$.
In this case $c$ is determined by $(a,k_0,k_1)$ such that $k_0+k_1=0$. Put $k_1=b$ (and then $k_0=-b$). 
Let us make the formulas for $k'_i$ explicit in this case. We have
$$
\left\{\begin{array}{rcl}
k'_0&=&-b+a/2,\\
k'_1&=&b-a/2,\\
k'_2&=&-b-a/2,\\
k'_3&=&b+a/2.
\end{array}\right.
$$w

\end{exemple}

\bigskip

\subsection{Recollection}
We will rewrite here the statement of Theorem~\ref{theo:fixed-quiver} 
with the help of the combinatorics developed in this section. 
First, we will denote by $\ZCB_{\!c}(l,1,n)$ the Calogero-Moser space 
associated with the complex reflection group $G(l,1,n)$ and with the 
parameter $c$. Here, we label $\CM^\times$-fixed points using Gordon's convention, which differs from our usual convention by a twist (see Proposition \ref{prop:compare-Gordon}). Recall that $(\CC_k)^l[\equiv n,\le n]$ is the set 
of $l$-tuples $\g=(\g^0,\dots,\g^{l-1})$ of $k$-cores 
such that 
$$|\g| \le n \text{~and~} |\g| \equiv n \mod k.$$
If $\l=(\l^0,\dots,\l^{l-1}) \in \PC^l[n]$ and $\Core_k(\l)=\gamma=(\gamma^0,\dots,\gamma^{l-1})$, then we define $\beta^\flat_{k,\gamma}(\l)$ to be the $m$-partition 
$\mu=(\mu^0,\mu^1,\dots,\mu^{m-1})$ defined by 
$$\beta_{k,\gamma^i}(\l^i)=(\mu^{i+(k-1)l},\dots\mu^{i+l},\mu^i)$$
for all $i \in \{0,1,\dots,l-1\}$. (In other words, we have $\beta^\flat_{k,\gamma}(\lambda)=(\beta_{k,\gamma}(\lambda^\flat))^\flat$.) Also, if $\l \in \PC^l[n]$, we denote by 
$\chi_\l^{(l)}$ the associated irreducible character of $G(l,1,n)$ 
and by $z_\l^{(l)}$ the $\CM^\times$-fixed point in $\ZCB_{\!c}(l,1,n)$ 
associated with $\chi_\l^{(l)}$ by~\cite[Corollary~5.8]{gordon}.

\bigskip

\begin{theo}\label{theo:recapitulatif}
Assume that $\ZCB_{\! c}(l,1,n)$ is smooth and 
let $(a,k_0,k_1,\dots,k_{l-1})$ be the family of parameters associated with $c$ as in~\S\ref{sub:quiver-cm}. 
Then $\ZCB_{\! c}(l,1,n)^{\mub_{kl}}$ is smooth (so its irreducible components coincide 
with its connected components) and:
\begin{itemize}
\itemth{a} There is a bijection $\g \mapsto \ZCB(\g)$ between the set 
$(\CC_k)^l[\equiv n,\le n]$ and the set of irreducible components of 
$\ZCB_{\! c}(l,1,n)^{\mub_{kl}}$ such that $z_\l^{(l)} \in \ZCB(\g)$ 
if and only if $\Core_k(\l)=\g$ (for any $\l \in \PC^l[n]$). 

\itemth{b} Let $\g \in (\CC_k)^l[\equiv n,\le n]$ and let $r=(n-|\g|)/k$. Then there is an isomorphism 
of varieties 
$$i_\g : \ZCB_{\! c'}(lk,1,r) \longiso \ZCB(\g),$$
where $c' : \Ref(G(lk,1,r)) \to \CM$ is the parameter associated with the family 
$(a',k_0',k_1',\dots,k_{kl-1}')$ such that
$$
\begin{cases}
a'=ka,\\
k_j' = k_{(j \mod l)} + a \left(\left\lfloor 
\DS{\frac{j-1}{l}} \right\rfloor - \DS{\frac{k-1}{2}} + k(d_{1-j}-d_{-j})\right),  
\qquad 1 \le j \le m \mbox{ and }k'_0=k'_m,
\end{cases}
$$
which satisfies $z_{(\b^\flat_{k,\g})^{-1}(\mu)}^{(l)}=i_\g\bigl(z_{\mu}^{(kl)}\bigr)$ 
for all $\mu \in \PC^{kl}[r]$. Here, $d=(d_i)_{i \in \ZM/kl\ZM}$ is defined by
$$d=\res_m(\nu)+r\d_m,$$
where $\nu =(\b_{l,\vide})^{-1}(\g)$ is an $m$-core.
\end{itemize}
\end{theo}

\bigskip

\begin{coro}\label{coro:amusant}
Let $\g \in (\CC_k)^l[\equiv n,\le n]$ and let $r=(n-|\g|)/k$. 
$$i_\g^* : \Zrm(\CM G(l,1,n)) \longto \Zrm(\CM G(kl,1,r))$$ 
be the (surjective) morphism of algebras defined by
$$i_\g^*(e_{\chi_\l})=
\begin{cases}
e_{\chi_{\beta^\flat_{k,\gamma}(\l)}} & \text{if $\Core_k(\l)=\g$,}\\
0 & \text{otherwise.}\\
\end{cases}
$$
Then $i_\g^*(\FC_i(\Zrm(\CM G(l,1,n)))) \subset \FC_i(\Zrm(\CM G(kl,1,r)))$ 
for all $i$.
\end{coro}

\bigskip

\begin{rema}\label{rem:amusant}
Note that the statement of Corollary~\ref{coro:amusant} involves 
only the character theory of groups of type $G(l,1,n)$ and 
the combinatorics of partitions. However, our proof relies on the 
geometric results on quiver varieties proved here as well as 
the computation of the equivariant cohomology of smooth Calogero-Moser 
spaces done in~\cite{BS}.

It is a natural question to ask whether such a result can be 
proved directly by combinatorial methods.\finl
\end{rema}

\bigskip

\section{Type $G_4$}\label{sec:g4}

\medskip

\boitegrise{{\bf Hypothesis.} {\it In this section, and only in this 
section, we assume that $V=\CM^2$, we identify $\Gb\Lb_\CM(V)$ with $\Gb\Lb_2(\CM)$, we fix a primitive 
third root of unity $\o$, we set 
$$s=\begin{pmatrix} \o & 0 \\ \o^2 & 1 \end{pmatrix} 
\qquad \text{and}\qquad 
t=\begin{pmatrix} 1 & -\o^2 \\ 0 & \o \end{pmatrix}
$$
and we assume that $W=\langle s,t\rangle$, so that $W$ is of type $G_4$ 
in Shephard-Todd classification. We denote by $(y_1,y_2)$ the canonical basis 
of $V$ and by $(x_1,x_2)$ its dual basis. We also set $\r=1+2\o$, 
so that $\r^2=-3$.}}{0.75\textwidth}

\bigskip

In this case, we will check Conjecture~F even if $\ZCB_{\! c}$ is not smooth 
(and for all elements $\s \in \Gb\Lb_\CM(V)$ normalizing $W$, as any such elements 
belong to $\CM^\times \cdot W$, as we will see in the proof below). 

\bigskip

\begin{theo}\label{theo:g4}
If $W$ is of type $G_4$, then Conjecture~F holds.
\end{theo}

\bigskip

\begin{proof}
Let $\s \in \Gb\Lb_2(\CM)$ be of finite order and normalizing $W$. The conjugacy class of $s$ is 
$$\{s,t,sts^{-1},tst^{-1}\},$$
and is the unique conjugacy class of reflections of determinant $\o$. Since $\s s \s^{-1}$ is a reflection 
of $W$ of determinant $\o$, there exists $w \in W$ such that $w\s$ commutes with $s$. 
Since $W$ acts trivially on $\ZCB_{\! c}$, we may, and we will, assume that $\s s \s^{-1} = s$. 
Now, $\s t \s^{-1}$ belongs to this conjugacy class and is different from $s$, 
so there exists $i \in \ZM$ such that $s^i \s$ commutes with $t$. So, replacing 
$\s$ by $s^i\s$, we may, and we will, assume that $\s$ commutes with $s$ and $t$. 
In other words, $\s$ is a root of unity. 

Since $\Zrm(W) \simeq \mub_2$, we may also assume that the order of $\s$ is even, 
equal to $2d$ for some $d \ge 2$ (because the case $d=1$ is obvious). 
Let $I_d$ denote the ideal of $\Zb_c$ generated by the $\s(z)-z$, where 
$z$ runs over a set of generators of $\Zb_c$. Then 
$$\CM[\ZCB_{\! c}^\s]=\Zb_c/\sqrt{I_d}.\leqno{(\clubsuit)}$$
Let $\z$ denote the root of unity by which $\s$ acts: then $\s(z)-z=(\z^l-1)z$ 
if $z$ is homogeneous of $\ZM$-degree $l$. This shows that $I_d$ is generated 
by the generators of $\Zb_c$ of degree not divisible by $2d$. 
As we will see below, there is a set of generators of $\Zb_c$ whose $\ZM$-degrees 
belong to the interval $[-6,6]$. So, if $d \ge 4$, then $\ZCB_{\! c}^\s=\ZCB_{\! c}^{\CM^\times}$, 
so it is a finite set 
and then Conjecture~F is checked in this case. So we may assume that $d \in \{2,3\}$. 

Note that there is only one $W$-orbit of reflecting hyperplanes. So, 
as in Example~\ref{ex:cyclique}, we set 
$$k_j=\frac{1}{l} \sum_{i=1}^{2} \o^{-i(j-1)} c_{s^i},$$
so that $k_0+k_1+k_2=0$. 

In~\cite{BT}, Thiel and the first author have provided algorithms for computing 
presentations of the algebra $\Zb_c$, and such algorithms have been included 
in the {\tt MAGMA} package {\tt CHAMP} developed by Thiel~\cite{thiel champ} 
(about the {\tt MAGMA} software, see~\cite{magma}). These algorithms provide 
an explicit presentation of the center $\Zb_c$ given as follows:
\begin{itemize}
\item[] Generators:

\centerline{$Y_1$, $Y_2$, $X_1$, $X_2$, $\euler$, $A$, $B$, $D$;}

\medskip

\item[] Relations:

\centerline{${\small \begin{cases}
(\Zrm 1) & 
\euler^4 + 4\r Y_1 X_1 - A B + 3\r \euler D+18 (k_1^2+k_1k_2+k_2^2) \euler^2 \\ & +756 (k_1^2k_2+k_1k_2^2) \euler=0\\
(\Zrm 2) & 
\r \euler^3 A - 4 X_1 \euler B - D A - \r Y_1 X_2 - 30\r (k_1^2+k_1k_2+k_2^2) \euler A \\
&-108\r (k_1^2k_2+k_1k_2^2) A=0 \\
(\Zrm 3) & 
Y_2 A - \euler B^2 + \r Y_1 D+54 (k_1^2+k_1k_2+k_2^2) Y_1 \euler-324 (k_1^2k_2+k_1k_2^2) Y_1=0 \\
(\Zrm 4) & 
\r \euler^3 B -4 X_1 Y_2 - \r Y_1 \euler A - B D-30\r (k_1^2+k_1k_2+k_2^2) \euler B\\
&-108\r (k_1^2k_2+k_1k_2^2) B=0 \\
(\Zrm 5) & 
\euler A^2 - 4 X_1 D - X_2 B+72 \r (k_1^2+k_1k_2+k_2^2) X_1 \euler-432 \r (k_1^2k_2+k_1k_2^2) X_1=0 \\
(\Zrm 6) & 
Y_2 \euler^3 + 3 Y_1^2 A  -3 Y_1 \euler^2 B- B^3 + 3 \r Y_2 D+144 (k_1^2+k_1k_2+k_2^2) Y_1 B\\
&+18 (k_1^2+k_1k_2+k_2^2) Y_2 \euler+756 (k_1^2k_2+k_1k_2^2) Y_2=0 \\
(\Zrm 7) & 
4 \r X_1 \euler^2 A - A^3 -16 X_1^2 B + X_2 \euler^3 +3\r X_2 D+18 (k_1^2+k_1k_2+k_2^2) X_2 \euler\\
&-192 \r (k_1^2+k_1k_2+k_2^2) X_1 A+756 (k_1^2k_2+k_1k_2^2) X_2=0 \\
(\Zrm 8) & 
\euler^6 +4 \r Y_1 X_1 \euler^2 +2 \r D \euler^3 - Y_2 X_2 + D^2-36 (k_1^2+k_1k_2+k_2^2) \euler^4\\
&+12 \r (k_1^2+k_1k_2+k_2^2) \euler D +216 \r (k_1^2k_2+k_1k_2^2) D+1080 (k_1^2k_2+k_1k_2^2) \euler^3\\
&+1620 (k_1^2+k_1k_2+k_2^2)^2 \euler^2-3888 (k_1^2k_2+k_1k_2^2) (k_1^2+k_1k_2+k_2^2) \euler \\
& -34992 (k_1^2k_2+k_1k_2^2)^2=0 \\
(\Zrm 9) & \r Y_1 A^2+4 X_1 B^2 +4 \euler^3 D-\r Y_2 X_2 +4 \r D^2-24 \r (k_1^2+k_1k_2+k_2^2) \euler^4\\
&-288 (k_1^2+k_1k_2+k_2^2) \euler D-576 (k_1^2+k_1k_2+k_2^2) Y_1 X_1 \\ 
&+2160 \r (k_1^2+k_1k_2+k_2^2)^2 \euler^2 +432 \r (k_1^2k_2+k_1k_2^2) \euler^3-864 (k_1^2k_2+k_1k_2^2) D\\
& +20736 \r (k_1^2k_2+k_1k_2^2) (k_1^2+k_1k_2+k_2^2) \euler+46656 \r (k_1^2k_2+k_1k_2^2)^2=0\\
\end{cases}}
$}
\end{itemize}
The generators have $\ZM$-degrees given by the following map
$$(Y_1,Y_2,X_1,X_2,\euler,A,B,D) \longmapsto (-4,-6,4,6,0,2,-2,0).$$
Moreover, in this presentation, $\CM[V]^W=\CM[X_1,X_2]$ and $\CM[V^*]^W=\CM[Y_1,Y_2]$.

We can deduce from this that $\ZCB_{\! c}$ is smooth if and only if 
$$k_0k_1k_2(k_0-k_1)(k_0-k_2)(k_1-k_2) \neq 0.\leqno{(\diamondsuit)}$$

\medskip

\noindent{\bfit First case: assume that $d=2$.} In this case, the above presentation shows that, whenever~($\diamondsuit$) holds, then 
$\ZCB_{\! c}^\s$ has four irreducible components $\XCB$, $p$, $q$ and $r$ where $\XCB$ 
has dimension $2$ and is isomorphic to 
$$\XCB \simeq \{(x,y,e)\in \CM^3~|~e(e-12k_0)(e-12k_1)(e-12k_2) = xy\},\leqno{(\heartsuit)}$$
and $p$, $q$ and $r$ are three points (which belong to $\ZCB_{\! c}^{\CM^\times}$). 
If~($\diamondsuit$) does not hold, then $\ZCB_{\! c}^\s$ still contains 
an irreducible component of dimension $2$ with the same description as above, 
and a few other isolated points (the number depends on the exact values of 
the $k_i$). By Remark~\ref{rem:stupide}, we only need to consider 
the irreducible component $\XCB$. But it follows from~($\heartsuit)$ and~Example~\ref{ex:cyclique} that then $\XCB$ is isomorphic 
to the Calogero-Moser associated with a cyclic group of order $4$ acting on a 
one-dimensional space: as there is an element $w$ of order $4$ in $W$, 
one can choose for $V'$ an eigenspace of $w$ and for $W'$ the cyclic group $\langle w\rangle$ and 
Conjecture~F is checked in this case (note also that the map $c \mapsto c'$ 
is linear).

\bigskip

\noindent{\bfit Second case: assume that $d=3$.} In this case, the above presentation 
shows that, whenever~($\diamondsuit$) holds, then 
$\ZCB_{\! c}^\s$ has two irreducible components $\XCB$ and $p$ where $\XCB$ 
has dimension $2$ and is isomorphic to 
$$\XCB \simeq \{(x,y,e)\in \CM^3~|~(e+6k_0)(e+6k_1)(e+6k_2)(e-12k_0)(e-12k_1)(e-12k_2) = xy\},\leqno{(\spadesuit)}$$
and $p$ is a point (which belongs to $\ZCB_{\! c}^{\CM^\times}$). 
If~($\diamondsuit$) does not hold, then $\ZCB_{\! c}^\s$ still contains 
an irreducible component of dimension $2$ with the same description as above, 
and maybe one extra point according to the values of the $k_i$. 
By Remark~\ref{rem:stupide}, we only need to consider 
the irreducible component $\XCB$. 
But it follows from~($\spadesuit)$ and Example~\ref{ex:cyclique} that then $\XCB$ is isomorphic 
to the Calogero-Moser associated with a cyclic group of order $6$ acting on a 
one-dimensional space: as $-s \in W$ is of order $6$, 
one can choose for $V'$ the $\o$-eigenspace of $s$ and for $W'$ the cyclic group $\langle -s\rangle$ and 
Conjecture~F is checked in this case (note also that the map $c \mapsto c'$ 
is linear).
\end{proof}

\end{document}